\documentclass[11pt]{amsart}
\usepackage{amssymb,amsmath,amsthm}
\usepackage{graphicx}
\usepackage{color}
\usepackage[usenames,dvipsnames,svgnames,table]{xcolor}
\usepackage{tikz}
\usepackage{comment}
\usepackage{epstopdf}
\usepackage{float}

\newtheorem{lemma}{Lemma}[section]
\newtheorem{theorem}[lemma]{Theorem}
\newtheorem{corollary}[lemma]{Corollary}
\newtheorem*{definition}{Definition}
\newtheorem*{remark}{Remark}

\renewcommand{\phi}{\varphi}
\renewcommand{\theta}{\vartheta}
\renewcommand{\epsilon}{\varepsilon}

\title[Multicorns and Unicorns II]{On Multicorns and Unicorns II: Bifurcations in Spaces of Antiholomorphic Polynomials}

\author[S. Mukherjee]{Sabyasachi Mukherjee}
\address{Jacobs University Bremen, Campus Ring 1,  Bremen 28759, Germany}
\email{s.mukherjee@jacobs-university.de, sabya@math.stonybrook.edu}

\author[S. Nakane]{Shizuo Nakane}
\address{Tokyo Polytechnic University, 1583, Iiyama, Atsugi, Kanagawa 243-0297, Japan}
\email{nakane@gen.t-kougei.ac.jp}

\author[D. Schleicher]{Dierk Schleicher}
\address{Jacobs University Bremen, Campus Ring 1, Bremen, 28759, Germany}
\email{d.schleicher@jacobs-university.de}

\date{\today}

\begin{document}

\begin{abstract}
The \emph{multicorns} are the connectedness loci of unicritical antiholomorphic polynomials $\bar{z}^d + c$. We investigate the structure of boundaries of hyperbolic components: we prove that the structure of bifurcations from hyperbolic components of even period is as one would expect for maps that depend holomorphically on a complex parameter (for instance, as for the Mandelbrot set; in this setting, this is a non-obvious fact), while the bifurcation structure at hyperbolic components of odd period is very different. In particular, the boundaries of odd period hyperbolic components consist only of parabolic parameters, and there are bifurcations between hyperbolic components along entire arcs, but only of bifurcation ratio $2$. We also count the number of hyperbolic components of any period of the multicorns. Since antiholomorphic polynomials depend only real-analytically on the parameters, most of the techniques used in this paper are quite different from the ones used to prove the corresponding results in a holomorphic setting.
\end{abstract}

\maketitle
\tableofcontents

\section{Introduction}
We consider the iteration of unicritical antiholomorphic polynomials $f_c(z) = \bar{z}^d + c$ for any degree $d \geq 2 $ and $c \in \mathbb{C}$ (any unicritical antiholomorphic polynomial can be affinely conjugated to an antiholomorphic polynomial of the form $f_c$). In analogy to the holomorphic case, the set of all points which remain bounded under all iterations of $f_c$ is called the \emph{filled-in Julia set} $K(f_c)$. The boundary of the filled-in Julia set is defined to be the \emph{Julia set} $J(f_c)$ and the complement of the Julia set is defined to be its \emph{Fatou set} $F(f_c)$. This leads, as in the holomorphic case, to the notion of \emph{Connectedness Locus} of degree $d$ unicritical antiholomorphic polynomials:

\begin{definition}
The \emph{multicorn} of degree $d$ is defined as $\mathcal{M}^{\ast}_d = \lbrace c \in \mathbb{C}: K(f_c)$ is connected$\rbrace$, where $f_c(z)=\bar{z}^d+c$. The multicorn of degree $2$ is called the \emph{tricorn}.
\end{definition}

The dynamics of quadratic antiholomorphic polynomials and their connectedness locus $\mathcal{M}_2^*$ were first studied in \cite{CHRS}. Milnor found small tricorn-like sets in the parameter space of real cubic polynomials \cite{Mi1}. Nakane \cite{Na1} proved that $\mathcal{M}^*_2$ is connected, in analogy to Douady and Hubbard's classical proof on the Mandelbrot set. This result can be generalized to multicorns of any degree. Later,  the structure of hyperbolic components of $\mathcal{M}_d^*$ was studied via the multiplier map (even period case) or via the critical value map (odd period case) \cite{NS}. These maps are branched coverings over the unit disk of degree $d-1$ and $d+1$ respectively, branched only over the origin. Quite recently, Hubbard and Schleicher \cite{HS} proved that the multicorns are not path connected, confirming a conjecture of Milnor.

\begin{figure}[h]
\centering{\includegraphics[width=6.8cm]{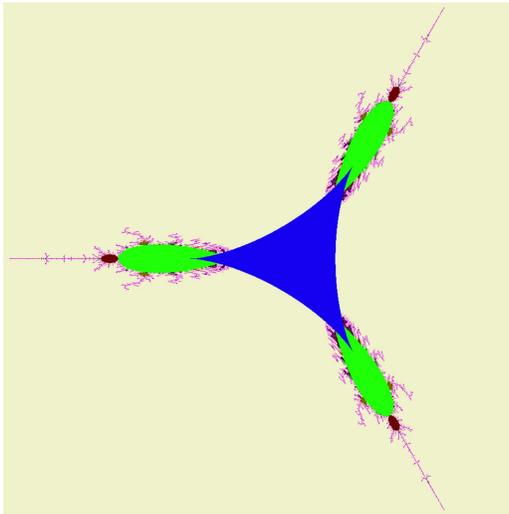}}
\caption{The tricorn, where the central deltoid region is its period $1$ hyperbolic component.} 
\label{FigConnLocus}
\end{figure}

The main purpose of this paper is to reveal the structure of the boundaries of the hyperbolic components and bifurcation phenomena. When the period is even, we show that the branched covering property of the multiplier map (which is real-analytic but not holomorphic) extends to the boundary as if it were holomorphic. This implies that bifurcations from the boundary of an even period component always occur at single parabolic parameters. This is the same as the unicritical polynomial family $p_c(z) = z^d + c : c \in \mathbb{C}$, see Figure \ref{evenandodd} (left). The following theorem, which is proved in Section \ref{SecIndiffDyn}, confirms this statement.

\begin{theorem}[Bifurcations From Even Periods]\label{ThmEvenBif} 
If a unicritical antiholomorphic polynomial $f_c$ has a $2k$-periodic cycle with multiplier 
$e^{2\pi ip/q}$ with $\mathrm{gcd}(p,q)=1$, then $c$ sits on the boundary of a hyperbolic component of 
period $2kq$ (and is the root thereof). 
\end{theorem}

On the other hand, the boundary of a hyperbolic component of odd period $k$ consists only of parabolic parameters of period $k$ and multiplier $+1$. Hence the bifurcation from an odd period component is quite restricted. Moreover, bifurcation occurs not at a point but along (part of) an arc. In fact, this phenomenon has already been observed in \cite{CHRS} in case $d = 2$ for the period one component. We will conceptualize, extend and refine their result in Section \ref{SecBifArcs}. See Figure \ref{evenandodd} (right), which is an enlargement of Figure \ref{FigConnLocus}. 

\begin{figure}[ht!] 
\begin{minipage}{0.49\linewidth}
\centering{\includegraphics[scale=0.26]{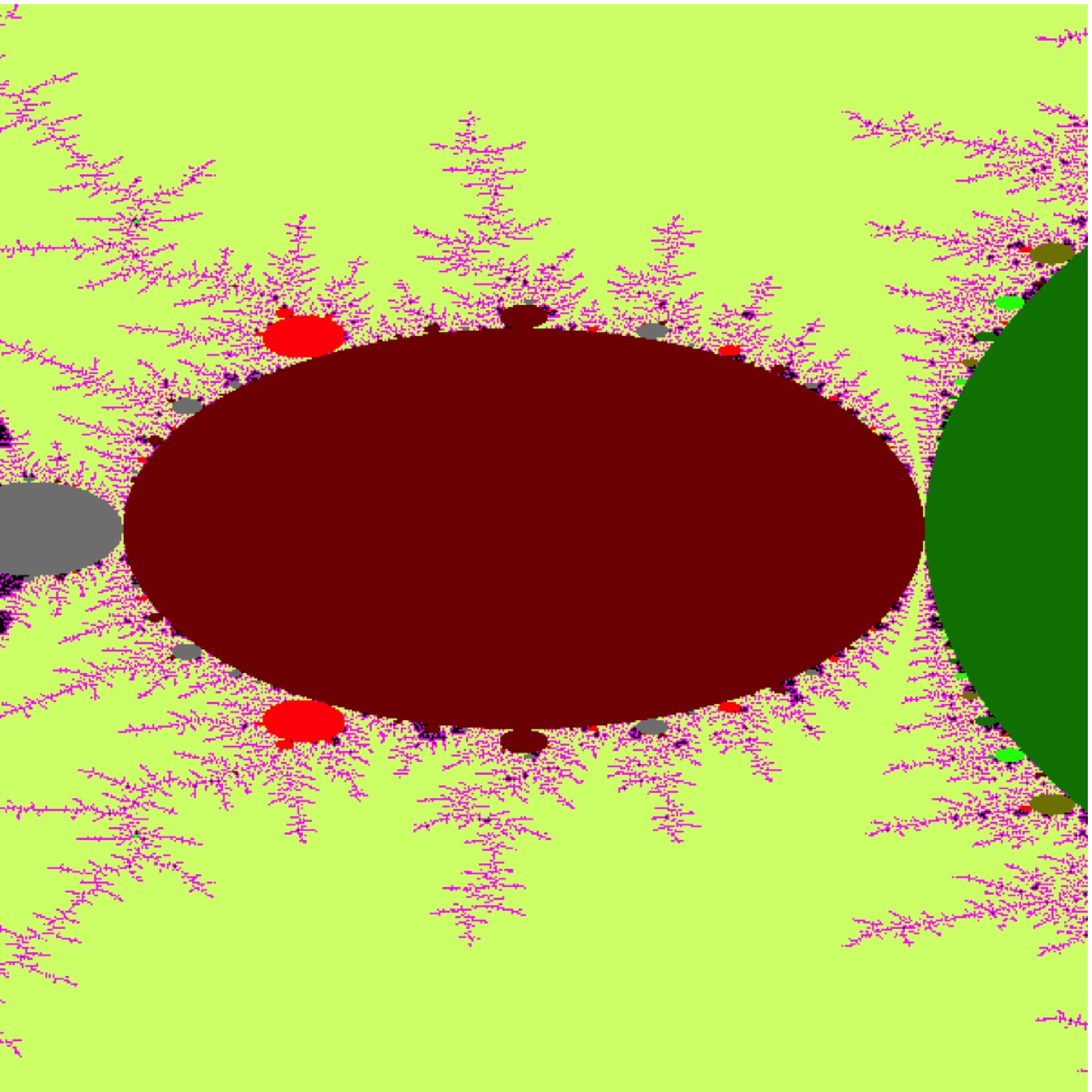}}
\end{minipage}
\begin{minipage}{0.49\linewidth} 
\centering{\includegraphics[scale=0.33]{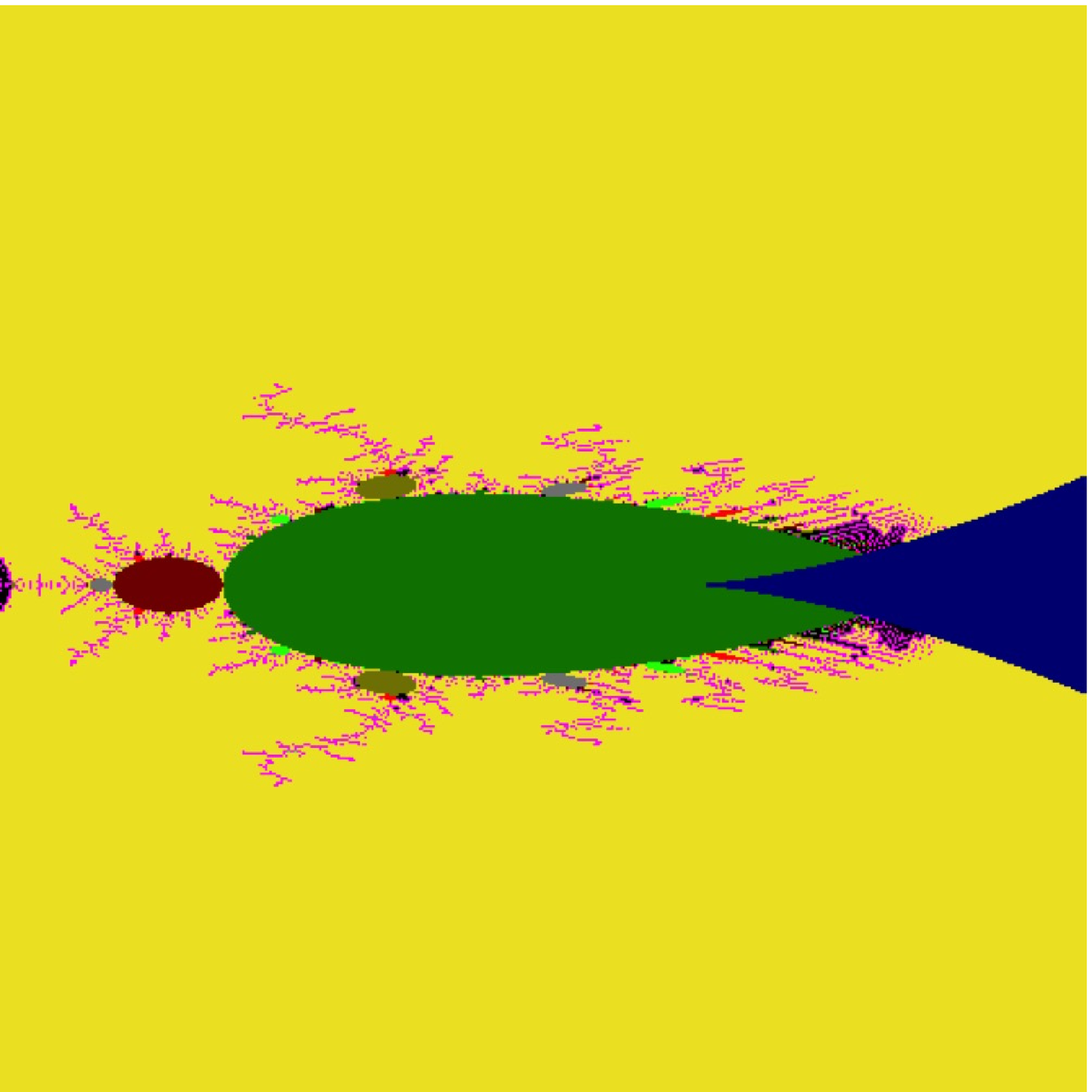}}
\end{minipage} 
\caption{Left: A period $4$ hyperbolic component of the tricorn; right: bifurcation along arcs from the period $1$ component to period $2$.} 
\label{evenandodd} 
\end{figure} 

It had been observed numerically long ago that the boundary of an odd period component consists of finitely many arcs and as many cusp points. In Section \ref{SecBifArcs} and Section \ref{SecOrbitPortraitsArcs}, we develop the techniques and notions required to give a rigorous description of these. The following theorem, which is proved in Section \ref{SecOddBdy}, can be viewed as a culmination point of these lines of ideas.

\begin{theorem}[Boundary Of Odd Period Components]\label{Exactly d+1}
The boundary of every hyperbolic component of odd period is a simple
closed curve consisting of exactly $d + 1$ parabolic cusp points as well as $d+1$ parabolic arcs, each connecting two parabolic cusps.
\end{theorem}

The proof of this fact uses combinatorial tools like orbit portraits (compare \cite{Mu}) and certain combinatorial rigidity results. In Section \ref{landingdiscont}, we utilize our work on orbit portraits and wake structures to prove a discontinuity of landing points of external dynamical rays, in contrast to the situation for the Mandelbrot set. This phenomenon is reminiscent of the parameter spaces of cubic (or higher degree) polynomials.

In Section \ref{No.ofHypComps}, we relate the number of hyperbolic components of a given period of the multicorns to the corresponding number for the multibrot sets (the connectedness loci of $p_c(z)=z^d+c$, denoted by $\mathcal{M}_d$). These numbers coincide for most periods, with exceptions only when the period $k$ is twice an odd number.

\begin{theorem}[Number Of Hyperbolic Components]\label{Numberhypcomp}
Denote the number of hyperbolic components of period $k$ in $\mathcal{M}_d^{\ast}$ (respectively $\mathcal{M}_d$) by $s^{\ast}_{d,k}$ (respectively $s_{d,k}$). Then, $s^{\ast}_{d,k} = s_{d,k}$ unless $k$ is twice an odd integer, in which case we have $s^{\ast}_{d,k} = s_{d,k} + 2 s_{d,k/2}$.
\end{theorem}

To illustrate the possible problems with antiholomorphic parameter spaces, consider the family of antiholomorphic quadratic polynomials $P_\mu(z)= \mu\bar{z}+\bar{z}^2$ which was discussed in the introduction of \cite{NS}; each $P_\mu$ is conformally conjugate to some $f_c$. As shown in Figure \ref{TricornCover}, the circle $|\mu|=1$ consists of maps with parabolic fixed points, and most parameters on this circle have neighborhoods in which every parameter has an attracting or indifferent fixed point: the open mapping principle for the multiplier map fails in this parametrization! This is related to a different problem of the parametrization: within the family $\lbrace f_c \rbrace_{c \in \mathbb{C}}$, each map $f_c$ is conformally conjugate to two other maps (with parameters $\zeta c$ and $\zeta^2c$, where $\zeta$ is any third root of
unity). However, in the family $\lbrace P_{\mu} \rbrace_{\mu \in \mathbb{C}}$, each map is conjugate to three or two further maps in the same family, depending on whether or not there is an attracting fixed point (in both cases, the exceptional case $c=0$ respectively\ $\mu=0$ behaves differently). 

Antiholomorphic polynomials of the form $q_\lambda(z)= \lambda(1+\bar{z}/d)^d$ form the ``true'' parameter space of our maps: every $f_c$ is conformally conjugate to one and only one antiholomorphic polynomial $q_\lambda$; it satisfies $\lambda = {d \bar{c}^d}/{c}$. The map $c\mapsto \lambda = {d \bar{c}^d}/{c}$ is a real-analytic branched cover of degree $d+1$, ramified only over $c = \lambda = 0$. We will use this parametrization in some of our proofs. The multicorn $\mathcal{M}_d^*$ has a $d+1$-fold rotational symmetry, and form a $d+1$-fold branched cover over the true parameter space. The quotient of $\mathcal M^*_d$ by this symmetry is thus naturally called the ``unicorn''. Pictorial illustrations of these fractals and their symmetries can be found in \cite{LS} and \cite{NS}.

We would like to thank John Hubbard, Hiroyuki Inou, Adam Epstein and Mitsuhiro Shishikura for fruitful discussions and useful advice. This work was partially supported by a grant from the Deutsche Forschungsgemeinschaft DFG, which we gratefully acknowledge. 

\section{Indifferent Dynamics}\label{SecIndiffDyn}
We start by observing that the polynomial-like maps developed by Douady and 
Hubbard~\cite{DH} make sense also in the antiholomorphic setting. 

\begin{definition}[Anti-polynomial-like Maps]\label{DefAntiPolyLike}
Let $U,V$ be simply connected domains in $\mathbb{C}$ such that 
$\overline{U} \subset V$. We call an antiholomorphic map $f :U \to V$ \emph{ 
anti-polynomial-like of degree $d$} if it is proper and has degree $d$. 
The \emph{filled-in Julia set} of $f$ is the set of all points in 
$U$ which never leave $U$ under iteration. 
\end{definition}

\begin{remark}
We define the degree of $f$ as the number of pre-images of any point, so it is always positive; equivalently, $d$ is the degree of the proper holomorphic map $f^{\ast}  \colon U \to V^{\ast}$ which is the complex conjugate of $f$.
\end{remark}

There is also an antiholomorphic analog to the Straightening Theorem 
\cite[Theorem~1]{DH} which can be proved in the same way as in the 
holomorphic case: every anti-polynomial-like map of degree $d$ is hybrid 
equivalent to an antiholomorphic polynomial of equal degree. 

The following theorem can be considered as a weak replacement, in certain 
cases,  for the open mapping principle of the multiplier; recall that this
is false in the parametrization $z \mapsto \mu \bar{z} + \bar{z}^2$!

\begin{theorem}[Indifferent Parameters on Boundary]\label{ThmIndiffBdyHyp} 
If $f_{c_0}(z) = \bar{z}^d+c_0$ has an indifferent periodic point of period $k$, then every neighborhood of $c_0$ contains parameters with attracting periodic points of period $k$, so the parameter $c_0$ is on the  boundary of a hyperbolic component of period $k$ of the multicorn $\mathcal{M}_d^*$. 

Moreover, every neighborhood of $c_0$ contains parameters for which all
period $k$ orbits are repelling. 
\end{theorem}

\begin{proof}
The idea of the proof is the same as the one for holomorphic polynomials by Douady~\cite[III.1]{Do}, so we give only a sketch. It is easier to prove this result in the family $q_{\lambda}(z)= \lambda (1 + \bar{z}/d)^d$ as described in the introduction. The map $f_{c_0}$ is conformally conjugate to $q_{\lambda_0}$ with $\lambda_0 = {d \overline{c_0}^d}/{c_0}$.

First we restrict $q_{\lambda_0}$ to an anti-polynomial-like map of degree $d$ and perturb it slightly so that the $k$-cycle becomes attracting (or repelling, so that all the other $k$-periodic orbits remain repelling); this requires a slight adjustment of the domain of the perturbed anti-polynomial-like map. The straightening theorem supplies an antiholomorphic polynomial of equal degree with a single critical point in $\mathbb{C}$ of maximal multiplicity, so it is conformally conjugate to a map $q_{\lambda^{\prime}}$ for a unique $\lambda^{\prime}\in\mathbb{C}$. The Beltrami differential in the straightening theorem can be chosen to be arbitrarily close to zero when the perturbation is small enough, so $\lambda^{\prime}$ can be chosen arbitrarily close to $\lambda$. This shows the result in the family $q_\lambda$, and it follows in the original family because the mapping $\lambda = d \bar{c}^d/c$ is a covering map of degree $d+1$.
\end{proof}

\begin{figure}[h]
\centering{\includegraphics[width=68mm]{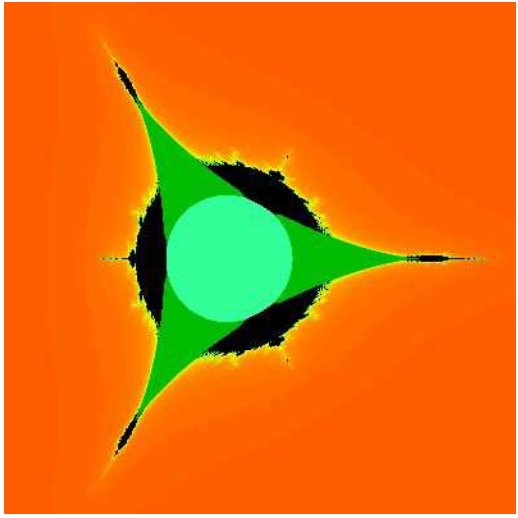}}
\caption{The connectedness locus of polynomials $\mu \bar{z}+\bar{z}^2$. The
circle in the center consists of parameters with parabolic fixed
points, and it intersects the boundaries of four distinct hyperbolic components of period
$1$: the disk inside, and three symmetric components outside of the
circle. In this space, Theorem \ref{ThmIndiffBdyHyp} is false.} 
\label{TricornCover}
\end{figure}

Now we prove Theorem \ref{ThmEvenBif} which shows the existence of bifurcations at parabolic parameters with even periodic parabolic cycles. This is subtler to prove than in the holomorphic case: if $f_{c_0}$ has an indifferent cycle of period $k$ and multiplier $\mu=e^{2\pi ip/q}$ with $p/q$ in lowest terms and $q\geq 2$, then in the holomorphic case it simply follows from the open mapping principle that $c_0$ is on the boundary of a hyperbolic component of period $kq$. In the antiholomorphic case, there is no open mapping principle, and the statement is not obvious. Our argument is based on similar ideas as for the proof of Theorem~\ref{ThmIndiffBdyHyp}. 

\begin{proof}[Proof of Theorem \ref{ThmEvenBif}]
Let $f_{c_0}(z)=\bar{z}^d+c_0$ be the given antiholomorphic polynomial and let $z_0$ be a 
parabolic $2k$-periodic point of $f_{c_0}(z)$. Since $f_{c_0}^{\circ 2k}$
is holomorphic, it makes sense to say that $z_0$ is the merger of a 
periodic point of period $2k$ and of points of period $2kq$. We will
estimate the  multiplier of the $2kq$-periodic orbit after perturbation of
$c_0$ to $c$. 

Let $\rho_0:=e^{2\pi ip/q}$. By assumption, we have for $z$ near $z_0$ 
\[ 
f_{c_0}^{\circ 2k}(z)  = z_0 + \rho_0(z-z_0) + O\left((z-z_0)^2\right) 
\,\,. 
\] 
Hence it follows from the Flower Theorem \cite[Section~10]{Mi3}: 
\[ 
f_{c_0}^{\circ 2kq}(z) = z + b(z-z_0)^{q+1} + O\left((z-z_0)^{q+2}\right) 
\,\,. 
\] 
The parabolic $2k$-cycle of $f_{c_0}$ is divided into two $k$-cycles of
$f_{c_0}^{\circ 2}$  and each $k$-cycle  contains one of the two critical
orbits of $f_{c_0}^{\circ  2}$ in its immediate basin.  Since this exhausts
all the critical orbits of $f_{c_0}^{\circ 2}$, there are exactly $q$
petals around $z_0$ and $b \neq 0$. 

Put $z_j = f_{c_0}^{\circ j}(z_0)$ for $0 \leq j \leq 2k-1$  and choose a 
fixed antiholomorphic polynomial $g$ satisfying $g(z_j) = 0$ for all $j$, $g'(z_j) = 0$ 
for $j \neq 1$, $g'(z_1)=-f'_{c_0}(z_1)$ and $g(z) = O(\bar{z}^d)$ as $z \to 
0$. Let $F_{\varepsilon} :=f_{c_0}+\varepsilon g$ for $\varepsilon \in \mathbb{C}$.  Then $(z_j)$ is still a 
$2k$-periodic cycle of $F_{\varepsilon}$. Let $\rho_{\varepsilon}$ be its multiplier. Then,
\begin{eqnarray*}  
F_{\varepsilon}^{\circ 2k}(z) 
&=& z_0 + \rho_{\varepsilon}(z-z_0) + O((z-z_0)^2) \,\,\mbox{ and } \\
F_{\varepsilon}^{\circ 2kq}(z) &=& z_0 + \rho_{\varepsilon}^q (z-z_0)  + b(z-z_0)^{q+1} \\  
&& + O\left(\varepsilon (z-z_0)^2\right) + O\left((z-z_0)^{q+2} \right) \,\,. 
\end{eqnarray*} 
We have
\begin{eqnarray*}
\rho_{\varepsilon} &=& \frac{\partial}{\partial z} F_{\varepsilon}^{\circ 2k}(z_0)  = 
\prod_{j=0}^{k-1}F'_{\varepsilon}(z_{2j+1})\overline{F'_{\varepsilon}(z_{2j})}\\  &=& 
\left(1+\varepsilon\frac{g'(z_1)}{f'_{c_0}(z_1)}\right) 
\prod_{j=0}^{k-1}f'_{c_0}(z_{2j+1})\overline{f'_{c_0}(z_{2j})}  = (1-\varepsilon)\rho_0 
\end{eqnarray*} 
and $\rho_\varepsilon^q = 1-q\varepsilon+O(\varepsilon^2)$; here $F'_\varepsilon$ denotes the 
antiholomorphic derivative $\frac{\partial}{\partial \bar{z}}F_\varepsilon$, 
and similarly for $f'_{c_0}$.

Any $2kq$-periodic point $z$ of $F_{\varepsilon}$ bifurcating from $z_0$ 
satisfies 
\begin{align} 
0 &= F_{\varepsilon}^{\circ 2kq}(z) - z 
\label{EqPeriod2qk} \\ 
&= z_0-z + \rho_\varepsilon^q(z-z_0) + b(z-z_0)^{q+1} 
+O\left(\varepsilon(z-z_0)^2\right) + O\left((z-z_0)^{q+2}\right) 
\nonumber \\
&= (z-z_0)\left(-1+1-q\varepsilon+O(\varepsilon^2)+b(z-z_0)^q \rule{0pt}{11pt}\right. 
\nonumber \\ 
&  +
\left.O(\varepsilon(z-z_0)) + O\left((z-z_0)^{q+1}\right) 
\rule{0pt}{11pt}\right) \,\,. 
\nonumber
\end{align} 
There are clearly $q$ such bifurcating points near $z_0$. 
To stress dependence on $\varepsilon$, we write $z_\varepsilon$ for $z$ and get 
$b(z_\varepsilon-z_0)^q=q\varepsilon(1+o(1))$ and hence 
\[ z_\varepsilon = z_0 + (q\varepsilon/b)^{1/q}\, 
\left(1 + o(1)\right) \,\,\mbox{ as }\,\, \varepsilon \to 0 \,\,. 
\] 
The multiplier is
\begin{eqnarray*}
\tilde\rho_{\varepsilon} &=& 
\left( \frac{\partial}{\partial z} F_{\varepsilon}^{\circ 2kq}\right) (z_{\varepsilon}) = 
\rho_\varepsilon^q+b(q+1)(z_\varepsilon-z_0)^q+o(\varepsilon) \\  
&=& 1-q\varepsilon+(q+1)q\varepsilon+o(\varepsilon)  = 1 + q^2\varepsilon + o(\varepsilon) \,\,. 
\end{eqnarray*}

For sufficiently small $\varepsilon<0$ we have $|\tilde\rho_{\varepsilon}| < 1$, so 
$z_{\varepsilon}$ is  an attracting periodic point of $F_{\varepsilon}$. 

The rest of the proof is as usual: restrict $f_{c_0}$ to an anti-polynomial-like map of degree $d$ and make $\varepsilon$ small enough so that 
after perturbation we still get an anti-polynomial-like map of degree $d$; 
straightening yields a unicritical antiholomorphic polynomial $\bar{z}^d+c(\varepsilon)$, and 
$c(\varepsilon)$ depends continuously on $\varepsilon$ with $c(0)=c_0$ (continuity is 
not in general true for straightening, but our construction assures that the 
quasiconformal gluing map has dilatation tending to zero as $\varepsilon\to 0$). 

Note that we have $q$ choices of the $q$-th root in $z_{\varepsilon}$, but the 
multiplier is independent of this choice (at least to leading
order). All $q$ choices of $z_\varepsilon$ will thus be attracting
simultaneously. Recall that $z_0$ had period $2k$. Since our polynomials
$f_{c(\varepsilon)}$ can only have a single non-repelling cycle, all these
attracting periodic points must be on the same cycle, which therefore has
period at least $2qk$; the period cannot  exceed $2qk$ because of
(\ref{EqPeriod2qk}). This completes the proof. 
\end{proof}

For odd $k$, bifurcations have an entirely different form; see Section~\ref{SecOddBdy}.

We will need the following result, which was proved in \cite{Na1}. 

\begin{theorem}[Nakane]\label{RealAnalUniformization}
The map $\Phi : \mathbb{C} \setminus \mathcal{M}^{\ast}_d \rightarrow \mathbb{C} \setminus \overline{\mathbb{D}}$, defined by $c \mapsto \phi_c(c)$ (where $\phi_c$ is the B\"{o}ttcher coordinate of $f_c$ near $\infty$) is a real-analytic diffeomorphism. In particular, the multicorns are connected.
\end{theorem}

The previous theorem also allows us to define parameter rays of the multicorns. 

\begin{definition}[Parameter Ray]
The \emph{parameter ray} at angle $\theta$ of the multicorn $\mathcal{M}^{\ast}_d$, denoted by $\mathcal{R}_{\theta}$, is defined as $\lbrace \Phi^{-1}(r e^{2 \pi i \theta}) : r > 1 \rbrace$, where $\Phi$ is the diffeomorphism given in Theorem \ref{RealAnalUniformization}.
\end{definition}

\begin{definition}[Accumulation Set of a Ray]\label{accumulation_set}
The accumulation set of a parameter ray $\mathcal{R}_{\theta}$ of $\mathcal{M}_d^*$ is defined as $L_{\mathcal{M}_d^*}(\theta) := \overline{\mathcal{R}_{\theta}} \bigcap \mathcal{M}_d^*$.
\end{definition}

\begin{definition}[External Angle of a Parameter]
Let $c \in \mathbb{C} \setminus \mathcal{M}_d^*$. Then the unique angle $t(c)$ such that $c \in \mathcal{R}_{t(c)}$ is called the \emph{external angle} of the parameter $c$. 
\end{definition}

\begin{lemma}[Landing of Dynamical Rays]\label{Lem:Landing of Dyn Rays}
For every map $f_c$ and every periodic angle $t$ (under the map $t\mapsto -dt$ $($mod $1)$) of some period $n$, the
dynamical ray $\mathcal{R}_{t}^c$ lands at a repelling periodic point of period dividing
$n$, except in the following circumstances:
\begin{itemize}
\item
$c \in \mathcal{M}_d^*$ and $\mathcal{R}_{t}^c$ lands at a parabolic cycle;
\item
$c \notin  \mathcal{M}_d^*$ is on the parameter ray at some angle $(-d)^kt$, for $k \in \mathbb{N}$.
\end{itemize}
Conversely, every repelling or parabolic periodic point is the landing point of at
least one periodic dynamical ray for every $c \in \mathcal{M}_d^*$.
\end{lemma}

\begin{proof}
This is well known in the holomorphic case, and can be proved analogously in the antiholomorphic setting. For the case $c \in \mathcal{M}_d^*$ (i.e. when the Julia set is connected), see \cite[Theorem 18.10]{Mi3}; if
$c \notin \mathcal{M}_d^*$ (i.e. when the Julia set is disconnected), see \cite[Expos{\'e} no.\ VIII, \S II, Proposition 1]{DH1}, \cite[Appendix A]{GM}. For the converse, see
\cite[Theorem 18.11]{Mi3}.
\end{proof}

We should note that for every periodic point $z_0$ of period $p$ of $f_{c_0}$ with multiplier $\lambda(c_0,z_0) \neq 1$, the orbit of $z_0$ remains locally stable under perturbation of the parameter. With some restrictions, this is even true for the dynamical rays of $f_{c_0}$ that land at $z_0$.

\begin{lemma}[Stability of Landing Rays]\label{l:preserving-portrait1}
Let $c_0$ be a parameter such that $f_{c_0} = \bar{z}^d + c$ has a repelling periodic point $z_0$ of period $p$ so that the dynamical rays $\mathcal{R}_t^{c_0}$ with $t \in A := \lbrace t_1,\cdots,t_v \rbrace$ $(v \geq 2)$ land at $z_0$. 

1) Then there exists a neighborhood $U$ of $c_0$, and a unique real-analytic function $z \colon U \to \mathbb{C}$
with $z(c_0)=z_0$ so that for every $c \in U$ the point $z(c)$ is a repelling
periodic point of period $p$ for $f_{c}$, and all the dynamical rays $\mathcal{R}_t^{c}$ with $t \in A$ land at $z(c)$.

2) Let $\lbrace n_1, n_2, \cdots, n_v\rbrace$ be the periods of the rays in $A$. Define $F$ to be the set of all parameters $c$ such that $f_c$ has a parabolic cycle, and some dynamical ray of period $n_i$ $(i \in \lbrace 1, 2, \cdots, v\rbrace)$ lands on the parabolic cycle of $f_c$. Let $U$ be an open, connected neighborhood of $c_0$ which is disjoint from $F$ and from all parameter rays at angles in $\displaystyle \tilde A:=\bigcup_{j\geq 0}(-d)^{j}A$. Then for every parameter $c \in U$, the dynamical rays $\mathcal{R}_t^c$ with $t \in A$ land at a common point.   
\end{lemma}

\begin{proof}
We work out the case when $p$ is even. When $p$ is odd, we simply need to work with the holomorphic first resturn map $f_{c_0}^{\circ 2p}$ of $z_0$.

1) Since $z_0$ is a repelling periodic point of $f_{c_0}$, the multiplier $\lambda(c_0,z_0) := \left(f_{c_0}^{\circ p}\right)'(z_0)$ is greater in $1$ in absolute value. By the Implicit Function Theorem, we can solve the equation $F(c,z) := f_c^{\circ p}(z) - z=0$ for the periodic point $z(c)$ (observe that $\frac{\partial F}{\partial z}\vert_{\left(c_0,z_0\right)}=\left(f_{c_0}^{\circ p}\right)'(z_0)-1=\lambda(c_0,z_0) -1 \neq 0$) as a real-analytic function of $c$ in a neighborhood $U$ of $c$. The multiplier $\left(f_{c}^{\circ p}\right)'(z(c))$ of $f_c$ at $z(c)$ is a real-analytic function as well, and hence is greater than $1$ in absolute value for $c$ close to $c_0$. This proves that the point $z_0$ can be continued real-analytically as a repelling periodic point $z(c)$ of $f_{c}$ in a neighborhood $U$ of $c_0$. Arguing as in \cite[Lemma B.1]{GM}, \cite[Lemma 2.2]{Sch}, one can show that the dynamical rays $\mathcal{R}_t^{c}$ with $t \in A$ land at the real-analytically continued periodic point $z(c)$ (possibly after shrinking $U$). 

2) Since $U$ does not intersect any parameter ray at angles in $\tilde{A}$, Lemma \ref{Lem:Landing of Dyn Rays} tells that for all $c \in U$, the dynamical rays $\mathcal{R}_t^c$ with $t \in A$ indeed land. Let $U^{\prime} := \lbrace c \in U:$ the dynamical rays $\mathcal{R}_t^c$ with $t \in A$ land at a common point$\rbrace$. For any $c \in U^{\prime}$, the common landing point of the dynamical rays $\mathcal{R}_t^c$ with $t \in A$ must be repelling as $U \cap F = \emptyset$. By part (1), we conclude that $U^{\prime}$ is open in $U$.

Let $\tilde{c}$ be a limit point of $U^{\prime}$ in the subspace topology of $U$. The dynamical rays $\mathcal{R}_t^{\tilde{c}}$ with $t \in A$ do land in the dynamical plane of $f_{\tilde{c}}$; we claim that they all land at a common point. Otherwise, at least two rays $\mathcal{R}_{t_1}^{\tilde{c}}$ and $\mathcal{R}_{t_2}^{\tilde{c}}$ (say) would land at two different repelling periodic points $z_1$ and $z_2$ (respectively) in the dynamical plane of $f_{\tilde{c}}$. Choose neighborhoods $V_1$ and $V_2$ of $z_1$ and $z_2$ respectively such that $V_1 \cap V_2 = \emptyset$. Applying part (1) to this situation, we see that for $c$ close to $\tilde{c}$, the landing points of $\mathcal{R}_{t_1}^{c}$ and $\mathcal{R}_{t_2}^{c}$ would lie in $V_1$ and $V_2$ respectively; i.e. the dynamical rays $\mathcal{R}_{t_1}^{c}$ and $\mathcal{R}_{t_2}^{c}$ would land at different points. This contradicts the fact that $\tilde{c}$ is a limit point of $U^{\prime}$. Hence, $\tilde{c} \in U'$, and $U^{\prime}$ is closed in $U$. Since $U$ is connected, $U^{\prime} = U$. This completes the proof of the lemma.
\end{proof}

\begin{corollary}[Stability of Rays at Bifurcations]\label{CorRayStability} 
Let, $c_0$ has an indifferent cycle of even period, and $H$ be a hyperbolic component with $c_0 \in \partial H$. Let, the dynamical rays $\mathcal{R}_{\theta}^{c_0}$ with $\theta \in \lbrace \theta_1, \cdots, \theta_r \rbrace$ land at a common repelling periodic point of $f_{c_0}$. Then for all $c \in H$, the dynamical rays $\mathcal{R}_{\theta}^{c}$ with $\theta \in \lbrace \theta_1, \cdots, \theta_r \rbrace$ land at a common repelling periodic point of $f_{c}$.

If further, the non-repelling cycle of $c_0$ is parabolic, and if the dynamical rays $\mathcal{R}_t^{c_0}$ with $t \in A := \lbrace t_1,\cdots,t_v \rbrace$ land at a parabolic periodic point $z_0$ of $f_{c_0}$, then for all $c \in H$, the dynamical rays $\mathcal{R}_t^{c}$ with $t \in A$ land at a common repelling cycle of $f_c$. The period of this repelling cycle may be equal or greater than the period of the parabolic cycle. If the periods are equal, then for all $c \in H$, the dynamical rays $\mathcal{R}_t^{c}$ with $t \in A$ land at a common repelling periodic point of $f_c$. 
\end{corollary}
\begin{proof}
Any arbitrarily small perturbation of $c_0$ intersect $H$, and hence by part (1) of Lemma \ref{l:preserving-portrait1}, the dynamical rays $\mathcal{R}_{\theta}^{c}$ with $\theta \in \lbrace \theta_1, \cdots, \theta_r \rbrace$ land at a common repelling periodic point of $f_{c}$ for any $c \in B(c_0, \varepsilon) \cap H$ (for $\varepsilon>0$ small enough). But the hyperbolic component $H$ does not intersect any parabolic parameter or any parameter ray. Therefore, by part (2) of Lemma \ref{l:preserving-portrait1}, the required property continues to hold throughout $H$.

Since the period of the indifferent cycle is even, the second fact can be proved as in the holomorphic setting. We refer the readers to \cite[Theorem 4.1, Lemma 4.4, Lemma 6.1]{Mi2} for the details.
\end{proof}

\begin{remark}
If the period of the parabolic cycle is odd, more care is needed. We will show later that the rays landing at an odd periodic parabolic point can be shared by two \emph{distinct} repelling cycles after perturbation.
\end{remark}

Before we state our next lemma, we need a brief digression to algebra. It is well-known that the \emph{resultant} of two univariate polynomials $P$ and $Q$ is a polynomial in the coefficients of $P$ and $Q$ which vanishes if and only if $P$ and $Q$ have a common root; i.e. if and only if $\mathrm{deg(gcd}(P,Q)) \geq 1$. It is sometimes desirable to generalize the notion of \emph{resultants} to be able to predict the exact degree of the gcd of $P$ and $Q$ in terms of their coefficients. This can be done by the so-called \emph{subresultants}, which are polynomials in the coefficients of $P$ and $Q$, and whose order of vanishing tells us the degree of the gcd of the two polynomials $P$ and $Q$. We refer the readers to \cite[\S 4.2.2]{BPC} for the precise definition of subresultants.

The main result on \emph{subresultants} that will be used in the proof of the Lemma \ref{LemIsolated} is the following:
 
\begin{lemma}\label{subresultants}
For two univariate polynomials $P$ and $Q$ over a domain (of degree $p$ and $q$ respectively, such that $p > q$), $\mathrm{deg(gcd}(P,Q)) \geq j$ for some $0 \leq j \leq q$ if and only if $sRes_0(P,Q)=\cdots =sRes_{j-1}(P,Q)=0$, where each $sRes_i(P,Q)$ is a polynomial expression in the coefficients of $P$ and $Q$.
\end{lemma} 

\begin{proof}
See \cite[Proposition 4.25]{BPC}.
\end{proof}

We denote the unit disc in the complex plane by $\mathbb{D}$ and the set of all roots of unity by $U$.

\begin{lemma}[Multipliers Isolated]\label{LemIsolated} 
Let $k \in \mathbb{Z}^+$ and $\mu \in \mathbb{D} \cup U$. Then the set of parameters $c\in\mathbb{C}$ for which $\bar{z}^d+c$ has a periodic orbit with period $2k$ and multiplier $\mu$ is finite. 
\end{lemma}
\begin{proof}
We embed the family $\{f_c\}$ in a two-parameter family of polynomials $\{P_{a,b}(z) = (z^d+a)^d+b;\,\,(a,b) \in \mathbb{C}^2\}$ which depends analytically on the parameters $a$ and $b$. The critical points of these polynomials are $0$ and the $d$ roots of the equation $z^d+a = 0$ such that each critical point has multiplicity $d-1$; however, these maps have only two critical values $a^d+b$ and $b$. Since $f_c^{\circ 2} = P_{\bar{c},c}$, our family can be regarded as a real 
$2$-dimensional slice (namely, $a = \bar{b}$) of the family $\{P_{a,b}\}_{a,b \in \mathbb{C}}$. A $2k$-cycle $\left( z_j\right)$ (where $z_j= f_c^{\circ j}(z_0), 0 \leq j\leq 2k-1$) of $f_c$ gives two distinct $k$-cycles $(z_{2i})$ and $(z_{2i+1})$ (for $i=0,\cdots,k-1$) of $f_c^{\circ 2}$; the multipliers of these two cycles are easily seen to be complex conjugates of each other; denote them by $\mu$ and $\overline{\mu}$. 

We define the following sets: 
\begin{itemize}
\item
$\Lambda := \lbrace c \in \mathbb{C} : f_c$ has a periodic orbit of period $2k$ with multiplier $\mu \rbrace$,
\item
$\Lambda(k,\mu)$ := $\lbrace (a,b) \in \mathbb{C}^2$ : $P_{a,b}$ has two distinct cycles of period dividing $k$ with multipliers $\mu$ and $\overline{\mu} \rbrace,$ 
\item
$\Lambda^{\prime}(k,\mu)$ = $\lbrace (a,b) \in \mathbb{C}^2$ :  $P_{a,b}$ has two distinct cycles of period $k$ with multipliers $\mu$ and $\overline{\mu} \rbrace$. 
\end{itemize}
 
Clearly, $\Lambda \subset \lbrace c \in \mathbb{C} : (\bar{c},c) \in \Lambda^{\prime}(k,\mu) \rbrace$,  $\vert \lbrace c \in \mathbb{C} : (\bar{c},c) \in \Lambda^{\prime}(k,\mu) \rbrace \vert = \vert \Lambda^{\prime}(k,\mu) \bigcap \lbrace (a,b) \in \mathbb{C}^2 : a = \bar{b} \rbrace \vert$ and $\Lambda^{\prime}(k,\mu) \subset \Lambda(k,\mu)$.
 
\emph{Case 1.} First assume that $\mu \in \left( \mathbb{D} \cup U \right) \setminus \mathbb{R}$. We will show that the set $\Lambda(k,\mu) $ is finite. Let us consider the complex numbers $a,b,z_0,z_1$ satisfying the four algebraic equations 
\begin{eqnarray} 
P_{a,b}^{\circ k}(z_0) - z_0 &=& 0, \quad
(P_{a,b}^{\circ k})^{\prime}(z_0) = \mu,
\label{EqBezout1}
\\ 
P_{a,b}^{\circ k}(z_1) - z_1 &=& 0, \quad 
(P_{a,b}^{\circ k})^{\prime}(z_1) = \overline{\mu}. 
\label{EqBezout2}
\end{eqnarray}
 
Let $R_{\mu}(a,b)$ (respectively $R_{\overline{\mu}}(a,b)$) be the resultant of the two polynomials $P_{a,b}^{\circ k}(z)-z$ and $(P_{a,b}^{\circ k})^{\prime}(z)-\mu$ (respectively $P_{a,b}^{\circ k}(z)-z$ and $(P_{a,b}^{\circ k})^{\prime}(z)-\overline{\mu}$). Then a solution $z_0$ of (\ref{EqBezout1}) (respectively $z_1$ of (\ref{EqBezout2})) exists if and 
only if $R_{\mu}(a,b) = 0$ (respectively $R_{\overline\mu}(a,b) = 0$). It follows that $\Lambda(k,\mu) = \{R_{\mu}(a,b) = R_{\overline{\mu}}(a,b) = 0\}$. Since $\mu \notin \mathbb{R}$, the two algebraic varieties $R_{\mu}(a,b) = 0$ and $R_{\overline\mu}(a,b) = 0$ are distinct.  By B\'{e}zout's theorem, the intersection of two algebraic varieties $R_{\mu}(a,b) = 0$ and $R_{\overline\mu}(a,b) = 0$ is either a finite set or a common irreducible component with unbounded projection over each variable. If they have a common irreducible component, say $S$, then $S \subset\Lambda(k,\mu)= \{R_{\mu}(a,b) = R_{\overline{\mu}}(a,b) = 0\}.$ This would force $\Lambda(k,\mu)$ to be unbounded. But since $\mu \in \mathbb{D} \cup U$, for any $(a,b) \in \Lambda(k,\mu)$, $P_{a,b}$ has two distinct attracting/parabolic cycles with multipliers $\mu$ and $\overline{\mu}$. As a result, these periodic orbits would attract at least one infinite critical orbit each. This forces all the critical orbits of $P_{a,b}$ to stay bounded. This implies that $\Lambda(k,\mu)$ is contained in the connectedness locus of the family of monic centered polynomials of degree $d^2$, which is compact by \cite{BH}: a contradiction!

\emph{Case 2.} Now we consider $\mu \in [-1 , 1)$. We no longer have two distinct algebraic varieties given by (\ref{EqBezout1}) and (\ref{EqBezout2}) and the arguments of the previous case do not work. To circumvent this problem, we use the theory of \emph{subresultants}.  

Note that for any $\left( a,b \right) \in \Lambda^{\prime}(k,\mu)$, the two polynomial equations $P_{a,b}^{\circ k}(z)-z=0$ and $(P_{a,b}^{\circ k})^{\prime}(z)-\mu=0$ (viewed as equations in $\mathbb{C}\left[ a,b \right] \left[ z \right]$) have exactly $2k$ distinct common solutions, all of which are simple roots of $P_{a,b}^{\circ k}(z)-z=0$ . Thus their gcd is a polynomial in $\mathbb{C}\left[ a,b \right] \left[ z \right]$ of degree $2k$. Note that in this case, the polynomial $P_{a,b}$ can have at most two cycles of multiplier $\mu$. It follows from Lemma \ref{subresultants} that:
\begin{align}
\label{subres1}
 \displaystyle \Lambda^{\prime}(k,\mu) = \lbrace \left( a,b \right) \in \mathbb{C}^2 : sRes_0\left( P_{a,b}^{\circ k}(z)-z , (P_{a,b}^{\circ k})^{\prime}(z)-\mu\right)=\cdots\\ 
 =sRes_{2k-1} \left(P_{a,b}^{\circ k}(z)-z , (P_{a,b}^{\circ k})^{\prime}(z)-\mu \right)=0 \rbrace. 
\nonumber
\end{align} 
It again follows from B\'{e}zout's theorem that the intersection of these algebraic varieties is a finite set: if there was a common irreducible component, then this component would be contained in $\Lambda^{\prime}(k,\mu)$ forcing it to be unbounded. But the two critical orbits of $P_{a,b}$ must be attracted by the two (attracting or parabolic) cycles of period $k$, implying that $\Lambda^{\prime}(k,\mu)$ is contained in the connectedness locus of the family of polynomials of degree $d^2$, which is compact by \cite{BH}: a contradiction!
 
\emph{Case 3.}  Finally we consider $ \mu = 1$. For any $\left( a,b \right) \in \Lambda^{\prime}(k,1)$, $P_{a,b}$ has two distinct parabolic cycles of period $k$ and multiplier $1$. The first return map $P_{a,b}^{\circ k}$ of each of these cycles fixes the petals. Since $P_{a,b}$ has only two infinite critical orbits and each cycle of petals absorbs at least one infinite critical orbit, there is exactly one petal associated with each of these parabolic periodic points. In other words, each of these parabolic periodic points is a double root of $P_{a,b}^{\circ k}(z)-z$. Hence, each of them contributes a linear factor to the gcd of $ P_{a,b}^{\circ k}(z)-z$ and $(P_{a,b}^{\circ k})^{\prime}(z)-1$. Since $P_{a,b}$ can not have any more parabolic cycles, the g.c.d of $ P_{a,b}^{\circ k}(z)-z$ and $(P_{a,b}^{\circ k})^{\prime}(z)-1$ is a polynomial of degree $2k$. Thus,
\begin{align}
\label{subres2}
\displaystyle \Lambda^{\prime}(k,1) \subset \lbrace \left( a,b \right) \in \mathbb{C}^2 : sRes_0\left( P_{a,b}^{\circ k}(z)-z , (P_{a,b}^{\circ k})^{\prime}(z)-1\right)= \cdots\\
=sRes_{2k-1} \left(P_{a,b}^{\circ k}(z)-z , (P_{a,b}^{\circ k})^{\prime}(z)-1 \right)=0 \rbrace.
\nonumber
\end{align}
By B\'{e}zout's theorem, this intersection of algebraic curves is either a finite set of points or contains an unbounded irreducible component. In addition to $\Lambda^{\prime}(k,1)$, the right hand side of (\ref{subres2}) can possibly contain some additional points of the form:

\emph{$S_1$} : Parameters $\left( a,b \right)$ such that $P_{a,b}$ has a single parabolic cycle of period $k^{\prime}$ (for some $k^{\prime}$ dividing $k$) and multiplier a $k/k^{\prime}$-th root of unity such that there are two cycles of petals for the parabolic cycle.

\emph{$S_2$}: Parameters $(a,b)$ such that $P_{a,b}$ has two distinct parabolic cycles of period $k^{\prime}$ and $k^{\prime\prime}$ (for some $k^{\prime}$ and $k^{\prime\prime}$ less than $k$ and dividing $k$) with multipliers some $k/k^{\prime}$-th and $k/k^{\prime\prime}$-th roots of unity respectively, such that each parabolic cycle has a single cycle of petals associated with it.

Observe that for any $(a,b) \in \Lambda^{\prime}(k,1) \cup S_1 \cup S_2$, the polynomial $P_{a,b}$ has both its critical orbits bounded. In other words, the right hand side of (\ref{subres2}) is contained in the connectedness locus of monic centered polynomials of degree $d^2$. Hence the intersection of the algebraic curves in (\ref{subres2}) can not contain an unbounded irreducible component and hence is finite. This finishes the proof of the lemma.
\end{proof}

\begin{remark} The restriction to even periods $k$ is essential: for odd $k$, the period 
$k$ orbit of the antiholomorphic polynomial $f_c$ does not split into two separate period $k$ 
orbits of $P_{\bar{c},c}$, so the given proof does not work. In fact, the set of 
parameters for which there is an orbit of odd period $k$ with multiplier 
$\mu=+1$ contains finitely many arcs: this set is equal to the union of the
boundaries of the hyperbolic components of period $k$; see  
Theorem~\ref{ThmParaArcs} below. 
\end{remark}

\begin{lemma}[Indifferent Dynamics of Odd Period]\label{LemOddIndiffDyn}  
The boundary of a hyperbolic component of odd period $k$ consists 
entirely of parameters having a parabolic cycle of period $k$. In 
local conformal coordinates, the $2k$-th iterate of such a map has the form 
$z\mapsto z+b z^{q+1}+\ldots$ with $q\in\{1,2\}$. 
\end{lemma}

\begin{proof} 
Let $H$ be a hyperbolic component of odd period, and $c \in \partial H$. Then, $f_c$ has an indifferent cycle of period $k'$ dividing $k$ (note that $k'$ is odd). The second iterate $f_c^{\circ 2k'}$ of the first return map $f_c^{\circ k'}$ of a periodic point of odd period $k'$ has a non-negative real multiplier. Since the cycle is indifferent, the absolute value of its multiplier is $+1$, and hence, the multiplier is always $+1$. This shows that every $c \in \partial H$ has a parabolic cycle of period $k'$ (dividing $k$) and of multiplier $1$.

The holomorphic first return map $f_c^{\circ 2k'}$ of such a parabolic periodic point has the local form: $z \mapsto z+b z^{q+1}+O(z^{q+2})$, where $q$ is the number of attracting petals attached to the parabolic periodic point \cite[\S 10]{Mi3}. Since the multiplier of the parabolic point is $+1$, $f_c^{2k'}$ fixes each petal (compare \cite[Corollary~4.2]{NS}), and hence there are precisely $q$ cycles of petals under $f_c^{\circ 2}$. Each of these cycles absorbs an infinite critical orbit of $f_c^{\circ 2}$. But $f_c^{\circ 2}$ has only two infinite critical orbits, and hence $q \in \lbrace 1,2\rbrace$.

If the period $k'$ of the parabolic cycle is strictly smaller than $k$ (i.e. if $H$ is a `satellite' component), then $c$ lies on the boundary of a hyperbolic component $H'$ of period $k'$ (by Theorem \ref{ThmIndiffBdyHyp}). So, exactly $k/k'$ points of the $k$-periodic attracting cycle of $H$ would coalesce (along with a repelling periodic point) at the parabolic parameter $c$, giving rise to $k/k'$ petals at each parabolic periodic point of $f_c$. In other words, the multiplicity of any parabolic point of $f_c$ would be $(k/k' + 1)$; i.e. $q=k/k'$. As $k/k'$ is necessarily odd, and $q \in \lbrace 1,2\rbrace$, we have $k/k'=1$, i.e. $k=k'$. This completes the proof of the lemma.
\end{proof}

\begin{definition}[Parabolic Cusps]\label{DefCusp}
A parameter $c$ will be called a {\em cusp point} (of odd period $k$) if $f_c$ has a parabolic 
periodic point of odd period $k$ such that $q=2$ in the previous lemma. 
\end{definition}

A cycle is parabolic if it is a merger of at least two periodic orbits, 
which generally is a co-dimension one condition. It turns out that for odd 
periods, this gives only a real co-dimension, so there are entire arcs of 
parabolic parameters (Theorem~\ref{ThmParaArcs}). Cusps are a merger of 
three periodic orbits, and this has real co-dimension two: there are only 
finitely many cusps. 

\begin{lemma}[Finitely Many Cusp Points] \label{LemCuspFin} 
The number of cusp points of any given (odd) period is finite. 
\end{lemma}
\begin{proof}
We use a similar idea as in the proof of Lemma~\ref{LemIsolated}. 
Let $k$ be odd and consider again the family $P_{a,b}(z)= (z^d+a)^d+b$ for 
complex parameters $a$ and $b$. Observe that if $c$ is a cusp parameter (with the period of the parabolic cycle being $k$), then the second iterate $f_c^{\circ 2}$ is a polynomial with a parabolic periodic point of period $k$ and multiplicity $2$ (in other words, $P_{\bar{c},c}$ has a double parabolic periodic point of period $k$). Thus, it suffices to prove the finiteness assertion in the family $P_{a,b}$. Suppose $a$, $b$, and $z$ be such that $z$ is a parabolic periodic point of period $k$ and multiplicity $2$ for $P_{a,b}$; then we have the following three equations: 
\begin{eqnarray*}
P^{\circ k}_{a,b}(z)-z=0 && \mbox {to have a fixed point} 
\\
(P^{\circ k}_{a,b})^{\prime}(z)=1 && \mbox{to make it parabolic} 
\\
(P^{\circ k}_{a,b})^{\prime \prime}(z)=0 &&\mbox{to make it a cusp.} 
\end{eqnarray*}
The set of simultaneous solutions to these three equations is, again using 
B\'{e}zout's Theorem, either finite or an algebraic variety with unbounded 
projection over every variable. 
 
The polynomial $P_{a,b}$ has two infinite critical orbits. If the map $P_{a,b}$ satisfies the previous three conditions, then $P_{a,b}$ has two distinct cycles of petals (recall that the multiplier of the parabolic cycle is $+1$), and each cycle of petals would attract an infinite critical orbit. Thus, both critical orbits would converge to the parabolic cycle; so they both would be 
bounded. Therefore, all the cusps are contained in the connectedness locus of monic centered polynomials of degree $d^2$, which is compact. Therefore, the set of solutions to the above three equations is finite. This proves the lemma.
\end{proof}

\begin{corollary}[Restricted Bifurcations from Odd Periods]\label{CorRestrictedOddBifurcations} 
If $f_c$ has an indifferent cycle of odd period $k$, and $c$ is on the 
boundary of a hyperbolic component (contained in $\mathcal{M}_d^*$) of period $n$, then $n\in\{k,2k\}$. 
\end{corollary}
\begin{proof}
If $c$ is on the boundary of a hyperbolic component $W$ of period $n$, then 
$n$ must be a multiple of $k$; setting $q:=n/k$, then $q$ points each of the $n$-periodic
attracting cycle in $W$ coalesce at $c$. The holomorphic second 
iterate $f_c^{\circ 2k}$ of the first return map $f_c^{\circ k}$ then has the form $z\mapsto z+z^{q+1}+\cdots$ 
in local coordinates, and $q\in\{1,2\}$ by Lemma~\ref{LemOddIndiffDyn}. 
\end{proof}

\begin{lemma}[Finitely Many Hyperbolic Components]\label{LemFinitelyHypComps} 
For every (even or odd) period $k$, the number of hyperbolic components 
of $\mathcal{M}_d^*$ with period $k$ is finite.
\end{lemma}
\begin{proof}
The structure theorem on hyperbolic components in \cite{NS} 
says that each even or odd period component has a unique center at which 
the critical orbit is periodic. Then, for the even period case, the claim
follows from Lemma~\ref{LemIsolated}. For the odd period case, 
we take the second iterate and it suffices to show that there are only 
finitely many parameters $(a,b)$ for which both critical orbits 
of $(z^d+a)^d+b$ are periodic with period $k$. The corresponding Julia sets
are connected,  so the claim follows once again from B\'{e}zout's theorem. 
\end{proof}

\begin{remark}
a) The finiteness can also be shown, possibly more dynamically, using Hubbard
trees. 

b) In Section \ref{No.ofHypComps}, we will give a recursive formula for the number of hyperbolic components of a given period.
\end{remark}

\begin{definition}[Characteristic Components and Points]
The \emph{characteristic Fatou component} of $f_c$ is defined as the unique Fatou component of $f_c$ containing the critical value $c$. If $f_c$ has a parabolic periodic orbit, then its \emph{characteristic parabolic point} is defined as the unique parabolic point lying on the boundary of the characteristic Fatou component.
\end{definition}

\begin{definition}[Roots and Co-Roots of Fatou Components]
Let $z$ be a boundary point of a periodic Fatou component $U$ corresponding to a (super-)attracting or parabolic unicritical antiholomorphic polynomial $f_c$ so that the first return map of $U$ fixes $z$. Then we call $z$ a \emph{root} of $U$ if it disconnects the filled-in Julia set; if it does not, we call it a \emph{co-root}.
\end{definition}

\begin{lemma}[Orbit Separation Lemma]\label{LemOrbitSeparation} 
For every antiholomorphic polynomial $f_c(z)=\bar{z}^d+c$ with parabolic dynamics, there are two periodic or pre-periodic dynamical rays which land at a
common point and which together separate the characteristic parabolic point from the rest of the parabolic cycle.
\end{lemma}
\begin{proof}
The proof is similar to the holomorphic case, see \cite[Lemma 3.7]{Sch}.
\end{proof}

\section{Bifurcations Along Arcs}\label{SecBifArcs}
In holomorphic dynamics, the local dynamics in attracting petals of parabolic periodic points is well-understood: there is a local coordinate $\zeta$ which conjugates the first-return dynamics to the form $\zeta\mapsto\zeta+1$ in a right half place (see Milnor~\cite[Section~10]{Mi3} or  Carleson-Gamelin~\cite[Section~II.5]{CG}). Such a coordinate $\zeta$ is called a \emph{Fatou coordinate}. Thus the quotient of the petal by the dynamics is isomorphic to a bi-infinite cylinder, called an \emph{Ecalle cylinder}. Note that Fatou coordinates are uniquely determined up to addition of a complex constant. 

In antiholomorphic dynamics, the situation is at the same time restricted and richer. Indifferent dynamics of odd period is always parabolic because for an indifferent periodic point of odd period $k$, the $2k$-th iterate is holomorphic with positive real multiplier, hence parabolic as described above. On the other hand, additional structure is given by the antiholomorphic intermediate iterate. 

\begin{lemma}[Antiholomorphic Fatou Coordinates]\label{PropFatouCoord} 
Suppose $z_0$ is a parabolic periodic point of odd period $k$ of $f_c$ 
with only one petal (i.e. $c$ is not a cusp) and $U$ is a periodic 
Fatou component with $z_0 \in \partial U$. Then there is an open subset 
$V \subset U$ with $z_0 \in \partial V$ and $f_c^{\circ k} (V) \subset V$ so 
that for every $z \in U$, there is an $n \in \mathbb{N}$ with $f_c^{\circ nk}(z)\in 
V$. Moreover, there is a univalent map $\Phi \colon V \to \mathbb{C}$ with 
$\Phi(f_c^{\circ k}(z)) = \overline{\Phi(z)}+1/2$, and $\Phi(V)$ contains a right 
half plane. This map $\Phi$ is unique up to horizontal translation. 
\end{lemma}

The map $\Phi$ will be called an  \emph{antiholomorphic Fatou coordinate} for 
the petal $V$; it satisfies $\Phi(f_c^{\circ 2k}(z))=\Phi(z)+1$ for $z \in V$ 
in accordance with the standard theory of holomorphic Fatou coordinates. 
This lemma applies more generally to antiholomorphic indifferent
periodic points such that the attracting petal has odd period. 

If $c$ is a cusp, the period of $U$ is even. Indeed, if $c$ is a cusp, then the parabolic periodic point $z_0$ has two petals, and hence is a dynamical root of $U$. By \cite[Corollary 4.2]{NS}, if the period of $U$ under $f_c$ were odd, then $z_0$ would lie on the boundary of only one Fatou component, which contradicts the fact that $z_0$ has two petals. It also follows from \cite[Corollary 4.2]{NS} that the period of $U$ under $f_c$ is $2k$. Therefore, the first return map of $U$ is $f_c^{\circ 2k}$, which is holomorphic, and the previous lemma does not apply.

\begin{proof} 
See \cite[Lemma 2.3]{HS}.
\end{proof}
 
The antiholomorphic iterate interchanges both ends of the Ecalle cylinder, so it must fix one horizontal line around this cylinder (the \emph{equator}). The change of coordinate has been so chosen that the equator is the projection of the real axis. We will call the vertical Fatou coordinate the \emph{Ecalle height}. Its origin is the equator. In particular, the Ecalle height of the critical value will be called the \emph{critical Ecalle height}. The existence of this distinguished real line, or equivalently an intrinsic meaning to Ecalle height, is specific to antiholomorphic maps and contributes, surprisingly enough, to some simplification of matters, as compared to the holomorphic case (it automatically relates the heights of the attracting  and repelling cylinders without a need for the \emph{horn map} from the repelling back into the attracting cylinder). One place where this has been exploited is in the proof of non-local connectivity and non-path connectivity  of the multicorns \cite{HS}.

\begin{theorem}[Parabolic Arcs]\label{ThmParaArcs} 
Let $c_0$ be a parameter such that $f_{c_0}$ has a parabolic cycle of odd period, and suppose that $c_0$ is not a cusp. Then $c_0$ is on a parabolic arc
in the  following sense: there  exists a simple real-analytic arc of parabolic
parameters 
$c(t)$ (for $t\in\mathbb{R}$) with quasiconformally equivalent but conformally 
distinct dynamics of which $c_0$ is an interior point.
\end{theorem}

This result has first been shown by Nakane \cite{Na2} using different methods.

\begin{proof}
We will use quasiconformal (qc-) deformations. The critical orbit is
contained in the parabolic basin. We parametrize the horizontal
coordinate within the Ecalle cylinder by $\mathbb{R}/\mathbb{Z}$.  Choose the horizontal
Fatou coordinate (the last degree of freedom of the  Fatou coordinate) so
that the critical value has real part $1/4$ within the Ecalle cylinder. Let, the critical Ecalle height of $f_{c_0}$ be $h$.

It is easy to change the complex structure within the Ecalle cylinder 
so that the critical Ecalle height becomes any assigned real value, 
for example via the quasiconformal homeomorphism 
$$
\ell_t : (x,y)\mapsto \left\{\begin{array}{ll}
                    (x,y+2tx) & \mbox{if}\ 0\leq x\leq 1/4 \\
                    (x,y+t(1-2x)) & \mbox{if}\ 1/4\leq x\leq 1/2 \\ 
                    (x,y-t(2x-1)) & \mbox{if}\ 1/2\leq x\leq 3/4 \\ 
                    (x,y-2t(1-x)) & \mbox{if}\ 3/4\leq x\leq 1 
                      \end{array}\right. 
$$
This homeomorphism $\ell_t$ commutes  with the map $I:z\mapsto\bar{z}+1/2$, 
hence the corresponding Beltrami form is invariant under the  map $I$. Note 
that $\ell_t(1/4,h) = (1/4,h+t/2)$. Translating the map $\ell_t$ by 
positive integers, we obtain a qc-map $\ell_t$ commuting with $I$ in a 
right half plane. 

By the coordinate change $z\mapsto \Phi_0(z)$ (where $\Phi_0$ is the attracting Fatou coordinate at the characteristic parabolic point of $f_{c_0}$), we can transport this Beltrami form from the right half plane into all the attracting petals, and it is forward invariant under $f_{c_0}$. It is easy to make it backward invariant by pulling it back along the dynamics. Extending it by the zero Beltrami form outside of the entire parabolic basin, we obtain an invariant Beltrami form, and the Measurable Riemann Mapping Theorem supplies a qc-map $\phi_t$ integrating this Beltrami form and conjugating the original map to a new antiholomorphic polynomial $f_{c(t)}$ within the same family. Its attracting Fatou coordinate at the characteristic parabolic point is given by $\Phi_t = \ell_t\circ\Phi_0\circ\phi_t^{-1}$. Thus the critical Ecalle height of $f_{c(t)}$ is $h+t/2$. 

Note that the Beltrami form depends real analytically on $t$, so the 
parameter $c(t)$ depends real analytically on $t$. We obtain a real
analytic map from $\mathbb{R}$ into the  multicorn $\mathcal{M}_d^{\ast}$. Since the critical
points of all $f_{c(t)}$ have  different Ecalle heights, which is a
conformal invariant, this map is injective. This proves the existence of an arc in the parameter plane with the required properties, and injectivity implies that the arc is simple.
\end{proof}

\begin{remark} 1. Since all the antiholomorphic polynomials $f_c$ on a parabolic arc $\mathcal{C}$ have quasi-conformally conjugate dynamics, there exists a fixed odd integer $k$ such that each $f_c$ on $\mathcal{C}$ has a $k$-periodic parabolic cycle. Such an arc will be referred to as a parabolic arc of period $k$.

2. By our construction, the critical Ecalle height of $f_{c(t)}$ is $h+t/2$. Therefore, the critical Ecalle height of $f_{c(t)}$ tends to $\pm \infty$ towards the ends of the arc; i.e. as $t \rightarrow \pm \infty$.

3. Numerical experiments suggest that the arc is a smooth curve in the parameter plane when parametrized by arclength. This would follow if we could prove that the map $c(t)$ had a nowhere vanishing derivative. One can, by passing to the biquadratic family, show that there are at most (possibly) finitely many singular points of this parametrization (see \cite{Mu1}). The question whether such finitely many exceptional points indeed exist, is related to the smoothness of certain algebraic curves and requires further investigation.
\end{remark}

\begin{lemma}[Endpoints of Parabolic Arcs]\label{PropArcEnds} 
For every parabolic arc $c(t)$, the limits $c^\pm:=\lim_{t\to\pm\infty} 
c(t)$ exist and are parabolic cusps of the same period. 
\end{lemma}
\begin{proof}
The parabolic arc $t\mapsto c(t)$ is a continuous map $\mathbb{R} \to \mathcal{M}_d^*$, so the 
limit sets $L^{\pm}$ of all accumulation points as $t\to\pm\infty$ are 
connected and non-empty. 

Let $k$ be the odd period of the parabolic cycle for all $c(t)$. 
By continuity, every $c\in L^{\pm}$ is parabolic with period at most $k$. 
By Lemma~\ref{LemOddIndiffDyn}, the period of the parabolic cycle of $c$ 
is exactly $k$.

Note that the critical Ecalle height depends continuously on the 
parameter. In fact, the construction of Fatou coordinates in 
lemma~\ref{PropFatouCoord} depends locally uniformly on the parameter 
around any non-cusp $c$. Therefore, if $c\in L^\pm$ is not cusp, 
the critical Ecalle heights for $c(t_n)$ tending to $c$, are bounded, 
which is a contradiction. Thus $L^\pm$ consists of cusps. By finiteness of 
the number of cusps, the claim follows. 
\end{proof}

\begin{lemma}[Parabolic Arcs Disjoint]\label{LemArcsDisjoint}
Two distinct parabolic arcs do not intersect.
\end{lemma} 
\begin{proof} 
Let $\mathcal{C}_1$ and $\mathcal{C}_2$ be two parabolic arcs, parametrized so that 
$c_i(t)$ has critical Ecalle height $t$ for every $t\in\mathbb{R}$. If these arcs 
have an interior point in common, then $c_1(t_0)=c_2(t_0)$ for some 
$t_0 \in \mathbb{R}$, and all $c_1(t)$ and all $c_2(t')$ are quasiconformally 
conjugate to $c_i(t_0)$ and hence to each other. For every $t \in \mathbb{R}$, 
the quasiconformal conjugation between $c_1(t)$ and $c_2(t)$ is conformal 
on the Ecalle cylinder (by identical critical Ecalle height) and hence 
on every bounded Fatou component, and it is also conformal on the basin 
of infinity. Since the Julia set has measure zero for every polynomial 
in which all critical orbits are in parabolic basins, $c_1(t)$ and $c_2(t)$ 
are conformally conjugate, and $c_1(t)=c_2(t)$ for all $t$ because
$c_1(t_0)=c_2(t_0)$. 
\end{proof}

However, the closures of two distinct parabolic arcs may intersect.

\begin{corollary}[Neighborhoods of Arcs]\label{CorArcNeighbors} 
For every parabolic arc $\mathcal{C}$ of odd period $k$, there is a unique hyperbolic component $H$ of period $k$ such that $\mathcal{C} \subset \partial H$. The arc $\mathcal{C}$ does not intersect the boundary of any other hyperbolic component of period $k$. 
\end{corollary} 
\begin{proof}
Let us assume that such an $H$ does not exist, and we will obtain a contradiction. By Theorem~\ref{ThmIndiffBdyHyp}, each point of $\mathcal{C}$ lies on the boundary of a hyperbolic component of period $k$ of $\mathcal{M}_d^*$. Since there are only finitely many hyperbolic components of period $k$, our assumption implies that there are distinct hyperbolic components $H_1, H_2, \cdots, H_r$ of period $k$ such that $\mathcal{C} \subset \partial H_1 \cup \partial H_2 \cup \cdots \cup \partial H_r$, but $\mathcal{C} \not\subset \partial H_i$ for $i=1, 2, \cdots, r$. It suffices to consider the case when $r=2$.

Let, $c: \mathbb{R} \rightarrow \mathcal{C}$ be the critical Ecalle height parametrization of $\mathcal{C}$ (given in Theorem \ref{ThmParaArcs}). Then there exists $t_0 \in \mathbb{R}$ such that $c \left((-\infty,t_0]\right) \subset \partial H_1$ and $c \left([t_0,\infty)\right) \subset \partial H_2$. Since $H_1$ is homeomorphic to $\mathbb{D}$ \cite[Lemma 5.4, Theorem 5.9]{NS}, and every parabolic limits at cusp points on both ends (by Lemma \ref{PropArcEnds}), $\partial H_1$ is a closed curve consisting of (possibly parts of) parabolic arcs and cusp points. Since each parabolic arc is a simple arc (i.e. there is no self-intersection), this implies that there exists a parabolic arc $\mathcal{C}'$ (intersecting $\partial H_1$), different from $\mathcal{C}$, such that the parameter $c(t_0) \in \mathcal{C} \cap \mathcal{C}'$. This contradicts Lemma \ref{LemArcsDisjoint}, and proves the existence of a hyperbolic component $H$ of period $k$ with $\mathcal{C} \subset \partial H$. 

The above argument also shows that if some $c \in \mathcal{C}$ lies on $\partial H$ and $\partial H'$, for two distinct hyperbolic components $H$ and $H'$ of period $k$, then $\mathcal{C} \subset \partial H \cap \partial H'$. But this is impossible because, by Theorem~\ref{ThmIndiffBdyHyp}, every neighborhood of every point on a parabolic arc meets parameters where all orbits of period $k$ are repelling. Therefore, $\mathcal{C}$ does not intersect the boundary of any hyperbolic component of period $k$, other than the boundary of $H$. The uniqueness follows. 
\end{proof}

\begin{definition}[Root Arcs and Co-Root Arcs]\label{DefRootArc}
We call a parabolic arc a \emph{root arc} if for every parameter $c$ on this arc, the parabolic cycle of $f_c$ disconnects the Julia set of $f_c$ (if $(z_i)_{1\leq i\leq k}$ is the parabolic cycle of $f_c$, then we say that $(z_i)_{1\leq i\leq k}$ disconnects $J(f_c)$ if $J(f_c)\setminus (z_i)$ is disconnected). Otherwise, we call it a \emph{co-root arc}.
\end{definition}

\begin{remark} Since the dynamics of all the points on the arc are quasiconformally 
conjugate, this classification of arcs is well-defined. 
\end{remark}

For an isolated fixed point $\hat{z}= f(\hat{z})$ where $f : U \rightarrow \mathbb{C}$ is a holomorphic function on a connected open set $U \subset \mathbb{C}$, the residue fixed point index of $f$ at $\hat{z}$ is defined to be the complex number

\begin{center}
$\displaystyle \iota(f, \hat{z}) = \frac{1}{2\pi i} \oint \frac{dz}{z-f(z)}$
\end{center}

where we integrate in a small loop in the positive direction around $\hat{z}$. If the multiplier $\rho:=f'(\hat{z})$ is not equal to $+1$, then a simple computation shows that $\iota(f, \hat{z}) = 1/(1-\rho)$. If $z_0$ is a parabolic fixed point with multiplier $+1$, then in local holomorphic coordinates the map can be written as $f(w) = w + w^{q+1} + \alpha w^{2q+1} + \cdots$ (putting $\hat{z}=0$), and $\alpha$ is a conformal invariant (in fact, it is the unique formal invariant other than $q$: there is a formal, not necessarily convergent, power series that formally conjugates $f$ to the first three terms of its series expansion). A simple calculation shows that $\alpha$ equals the parabolic fixed point index. The `r{\'e}sidu it{\'e}ratif' of $f$ at the parabolic fixed point $\hat{z}$ of multiplier $1$ is defined as $\left(\frac{q+1}{2} - \alpha\right)$. It is easy to see that the fixed point index does not depend on the choice of complex coordinates, and is a conformal invariant (compare \cite[\S 12]{Mi3}). 

The typical structure of the boundaries of hyperbolic components of even periods is that there are $d-1$ isolated root or co-root points which are 
connected by curves off the root locus. For hyperbolic components of odd periods, the story is in a certain sense just the opposite: the analogues 
of roots or co-roots are now arcs, of which there are $d+1$, and the analogues of the connecting curves are the isolated cusp points between 
the arcs; see Theorem~\ref{Exactly d+1}. There is trouble, of course, where components of even and odd periods meet, and we
get bifurcations along arcs: the root of the even period  component stretches along parts of two arcs. This phenomenon was first observed  in
\cite{CHRS} for the main component of the tricorn. The precise statement is given in the following result, which were proved in \cite[Proposition 3.7, Theorem 3.8, Corollary 3.9]{HS}. The proof of this fact uses the concept of holomorphic fixed point index. The main idea is that when several simple fixed points merge into one parabolic point, each of their indices tends to $\infty$, but the sum of the indices tends to the index of the resulting parabolic fixed point, which is finite.

\begin{theorem}[Bifurcations Along Arcs]\label{ThmBifArc}
Along any parabolic arc of odd period, the fixed point index is a real valued real-analytic function that tends to $+\infty$ at both ends. Every parabolic arc of period $k$ intersects the boundary of a hyperbolic component of period $2k$ at the set of points where the fixed-point index is at least $1$, except possibly at (necessarily isolated) points where the index has an isolated local maximum with value $1$. In particular, every parabolic arc has, at both ends, an interval of positive length at which a bifurcation from a hyperbolic component of odd period $k$ to a hyperbolic component of period $2k$ occurs.
\end{theorem}

\section{Parabolic Arcs and Orbit Portraits}\label{SecOrbitPortraitsArcs}
In this section, we study the combinatorial properties of the parabolic arcs in details. We begin with some definitions.

\begin{definition}[Orbit Portraits]
Let $\mathcal{O} = \left( z_1, z_2, \cdots, z_p \right)$ be a repelling or parabolic periodic orbit of some unicritical antiholomorphic polynomial $f_c$, and that a dynamical ray $\mathcal{R}_t^c$ at a rational angle $t \in \mathbb{Q}/\mathbb{Z}$ lands at some $z_i$. Let $\mathcal{A}_j$ be the set of angles of all dynamical rays landing at $z_j$, $j = 1, 2, \cdots, p$. By our assumption, each $\mathcal{A}_j$ is a non-empty finite subset of $\mathbb{Q}/\mathbb{Z}$. The collection $\lbrace \mathcal{A}_1 , \mathcal{A}_2, \cdots, \mathcal{A}_p \rbrace$ will be called the \emph{Orbit Portrait} $\mathcal{P(O)}$ of the orbit $\mathcal{O}$ corresponding to the antiholomorphic polynomial $f_c$.
\end{definition}

We refer the readers to \cite{Mu} for the basic properties of orbit portraits. The following definition follows from \cite[Lemma 2.8]{Mu}, where more information on these combinatorial objects can be found. 

\begin{definition}[Characteristic Arcs and Angles]
Let $\mathcal{P}$ be an orbit portrait associated with a repelling or parabolic periodic orbit of some unicritical antiholomorphic polynomial $f_c$. Among all the complementary arcs of the various $\mathcal{A}_j$, there is a unique one of minimum length. It is a critical value arc for some $\mathcal{A}_j$, and is strictly contained in all other critical value arcs. This shortest arc is called the \emph{characteristic arc} of $\mathcal{P}$, and the two angles at the ends of this arc are called the \emph{characteristic angles} of $\mathcal{P}$. 
\end{definition}

The characteristic angles, in some sense, are crucial to the understanding of orbit portraits.

\begin{lemma}[Number of Root and Co-roots]\label{number_root_co_root}
1) Every dynamical co-root is the landing point of exactly one dynamical ray, and this ray has the same period as the Fatou component.

2) Every periodic Fatou component (for a unicritical antiholomorphic polynomial $f_c$) of period greater than 1 corresponding to an attracting/parabolic cycle has exactly one root. If the period of the component is even; then it has exactly $d-2$ co-roots and if the period is odd; it has exactly $d$ co-roots. Every Fatou component of period 1 has exactly $d+1$ co-roots and no root.
\end{lemma}
\begin{proof}
See \cite[Lemma 3.4]{NS}.
\end{proof}

\begin{lemma}[Disjoint Closures of Fatou Components]\label{primitive}
Let, $H$ be a hyperbolic component of odd period $k$, and $c \in H$. Then, any two bounded Fatou components of $f_{c}$ have disjoint closures. Moreover, the dynamical root of the characteristic Fatou component of $f_c$ is the landing point of exactly two dynamical rays, each of period $2k$, and these are permuted by $f_c^{\circ k}$.
\end{lemma}

\begin{proof}
Let, $U_1$ and $U_2$ be two distinct bounded Fatou components of $f_{c}$ with $\overline{U_1}\cap \overline{U_2} \neq \emptyset$. By taking iterated forward images, we can assume that $U_1$ and $U_2$ are periodic. Then the intersection consists of a repelling periodic point $x$ of some period $n$. 

Each periodic (bounded) Fatou component of $f_c$ has period $k$. Since $x \in \partial U_1 \cap \partial U_2$, it follows that $f_c^{\circ k}(x) \in \partial U_1 \cap \partial U_2$. But the closures of two distinct bounded Fatou components of an antiholomorphic polynomial may intersect only at a single point. So, $f_c^{\circ k}(x)=x$. Therefore, the first return map of $\overline{U_1}$ (and of $\overline{U_2}$) fixes $x$, and $x$ disconnects the Julia set; hence, $x$ is the root of $U_1$ (as well as of $U_2$). It follows from \cite[Corollary 4.2]{NS} that $n=k$. But this contradicts the fact that every periodic Fatou component of $f_{c}$ has exactly one root, and there is exactly one cycle of periodic (bounded) Fatou components of $f_{c}$. This proves the first statement of the lemma. The statement about the dynamical rays landing at the characteristic Fatou component of $f_c$ now directly follows from \cite[Corollary 4.2]{NS}.
\end{proof} 

In what follows, we delve deep into the behavior of orbit portraits on parabolic arcs and use them to deduce many facts about the parameter rays of the multicorns.

Let $H$ be a hyperbolic component of odd period $k(\neq 1)$ with center $c_0$. The characteristic Fatou component $U_{c_0}$ of $f_{c_0}$ has exactly $d$ dynamical co-roots and exactly $1$ dynamical root on its boundary (in fact, this is true for every parameter $c \in H$). Let $S =\lbrace \theta_1, \theta_2, \cdots, \theta_d\rbrace$ (each angle of period $k$) be the set of angles of the dynamical rays that land at the $d$ dynamical co-roots on $\partial U_{c_0}$, and let $S' =\lbrace \alpha_1, \alpha_2\rbrace$ (each angle of period $2k$) be the set of angles of the dynamical rays that land at the unique dynamical root on $\partial U_{c_0}$. We will use these terminologies in the next two lemmas, the first one of which shows that the patterns of ray landing at the dynamical co-roots and root remain stable throughout the hyperbolic component.

\begin{lemma}\label{stable in hyp comp}
For any $c \in H$, the set of angles of the dynamical rays of $f_c$ that land at the $d$ dynamical co-roots on $\partial U_{c}$ (where $U_c$ is the characteristic Fatou component of $f_c$) is $S$, and the set of angles of the dynamical rays of $f_c$ that land at the unique dynamical root on $\partial U_{c}$ is $S'$.
\end{lemma}

\begin{proof}
$f_{c_0}$ has a unique super-attracting orbit, and all other periodic orbits of $f_{c_0}$ are repelling. Therefore, the dynamical rays $\mathcal{R}_t^{c_0}$ with $t \in S\cup S'$ land at repelling periodic points. The result now follows from Lemma \ref{l:preserving-portrait1} since $H$ does not intersect any parabolic parameter or any parameter ray.

We outline an alternative proof of this lemma making use of holomorphic motions. We embed our family $\displaystyle \lbrace f_c\rbrace_{c \in \mathbb{C}}$ in the family of holomorphic polynomials $\lbrace P_{a,b}(z)=(z^d+a)^d+b,\ a,\ b \in \mathbb{C}\rbrace$ (recall that $f_c^{\circ 2}=P_{\bar{c},c}$). The family $P_{a,b}$ depends holomorphically on the complex parameters $a$ and $b$. Let $\tilde{H}$ be the unique hyperbolic component of the family $\lbrace P_{a,b}\rbrace$ that contains the parameter $(\overline{c_0},c_0)$. Clearly, $\tilde{H}$ contains an open neighborhood of $H$ in $\mathbb{C}^2$. For each $(a,b) \in \tilde{H}$, the Julia set $J(P_{a,b})$ is connected, and hence, the B{\"o}ttcher coordinate $\phi_{a,b}$ of $P_{a,b}$ yields a biholomorphism from the basin of infinity $\mathcal{A}_{\infty}(P_{a,b})$ onto $\hat{\mathbb{C}}\setminus \overline{\mathbb{D}}$ conjugating $P_{a,b}$ to $z^{d^2}$. Using this, one can define a holomorphic motion $\Pi: \tilde{H}\times\mathcal{A}_{\infty}(P_{\overline{c_0},c_0})\rightarrow \mathbb{C}$ by requiring $\Pi((a,b),\ z)= \phi_{a,b}\circ\phi_{\overline{c_0},c_0}^{-1}(z)$ (compare \cite[Theorem 6.4]{Zi} for a similar construction). Using the $\lambda$-lemma \cite{MSS} (note that this holds when the parameter space of the holomorphic motion is biholomorphic to the unit ball in $\mathbb{C}^2$, and this is enough for our purpose), one sees that this holomorphic motion extends to $\overline{\mathcal{A}_{\infty}(P_{\overline{c_0},c_0})}$ yielding a topological conjugacy between $P_{\overline{c_0},c_0}$ on $\overline{\mathcal{A}_{\infty}(P_{\overline{c_0},c_0})}$ and $P_{a,b}$ on $\overline{\mathcal{A}_{\infty}(P_{a,b})}$ such that the conjugacy sends $\mathcal{R}_t^{(\overline{c_0},c_0)}$ to $\mathcal{R}_t^{(a,b)}$, for all $t \in \mathbb{R}/\mathbb{Z}$ and for all $(a,b)\in \tilde{H}$. Clearly, the property of being a dynamical co-root or root on the boundary of the characteristic Fatou component is also preserved by this topological conjugacy. Since for all $c \in H$, the characteristic Fatou component $U_{c}$ of $f_{c}$ has exactly $d$ dynamical co-roots and exactly $1$ dynamical root on its boundary, we have thus recovered all the co-roots and roots, and the above argument shows that the set of angles of the dynamical rays of $f_c$ that land at the $d$ dynamical co-roots on $\partial U_{c}$ is $S$, and the set of angles of the dynamical rays of $f_c$ that land at the unique dynamical root on $\partial U_{c}$ is $S'$.
\end{proof} 

By Lemma \ref{LemOddIndiffDyn}, $\partial H$ consists of parabolic arcs and cusps of period $k$.

\begin{lemma}\label{transfer}
Let $H$ be a hyperbolic component of odd period $k$ with center $c_0$. Let $c \in \partial H$, and the dynamical ray $\mathcal{R}_{\theta}^c$ lands at the characteristic parabolic point of $f_c$. 
\begin{enumerate}
\item If $c$ lies on a co-root arc on $\partial H$, then $\theta \in S$.

\item If $c$ lies on a root arcs on $\partial H$, then $\theta \in S'$.

\item If $c$ is a cusp point on $\partial H$, then $\theta \in S \cup S'$.
\end{enumerate}
\end{lemma}

\begin{proof}
Let $z_c$ be the characteristic parabolic point of $f_c$. Since the dynamical ray $\mathcal{R}_{\theta}^c$ lands at the $k$-periodic point $z_c$, by \cite[Lemma 2.5]{Mu}, the period of $\theta$ (under multiplication by $-d$) is either $k$ or $2k$.

1) Since $c$ lies on a co-root arc, $z_c$ is a dynamical co-root of the characteristic Fatou component (of period $k$) of $f_c$. By \cite[Lemma 3.4, Remark 3.5]{NS}, the period of $\theta$ must be $k$, and $\mathcal{R}_{\theta}^c$ is the unique dynamical ray of $f_{c}$ that land at $z_{c}$. Let $c'\in H$ be sufficiently close to $c$. The characteristic parabolic point $z_c$ of $f_c$ splits into an attracting periodic point of period $k$ and a repelling periodic point (say, $z_{c'}$) of period $k$ of $f_{c'}$. Furthermore, $z_{c'}$ is a dynamical co-root of the characteristic Fatou component of $f_{c'}$. Therefore, the dynamical ray $\mathcal{R}_{\theta_i}^{c'}$ (for some $\theta_i \in S$) is the unique dynamical ray of $f_{c'}$ that land at $z_{c'}$. We will now show that the dynamical ray $\mathcal{R}_{\theta_i}^c$ of $f_c$ lands at $z_c$. This would prove that $\theta=\theta_i \in S$.

If the dynamical ray $\mathcal{R}_{\theta_i}^c$ landed at a repelling periodic point of $f_c$, it would continue to land at the continuation of this repelling periodic point for nearby parameters (which would necessarily be different from a dynamical co-root born out of the parabolic point $z_c$), which is a contradiction. Therefore, $\mathcal{R}_{\theta_i}^c$ must land at a parabolic periodic point of $f_c$. If $\mathcal{R}_{\theta_i}^c$ landed at a non-characteristic parabolic point of $f_{c}$, the orbit separation lemma (Lemma \ref{LemOrbitSeparation}) would supply a partition of the dynamical plane by a pair rational dynamical rays landing at a common point, stable under small perturbations, separating the critical value $c$ from the dynamical ray $\mathcal{R}_{\theta_i}^c$. But for arbitrarily close parameters in $H$, the dynamical ray at angle $\theta_i$ lands on the boundary of the characteristic Fatou component, and there does not exist any rational dynamical ray pair separating the $\theta_i$ ray from the critical value, which is again a contradiction. Thus, the dynamical rays $\mathcal{R}_{\theta_i}^{c}$ land at the characteristic parabolic point $z_c$.

2) This is completely analogous to case (1).
 
3) Perturbing a cusp point $c$ into $H$ breaks the parabolic cycle into two disjoint repelling cycles (and an attracting one), and if the dynamical ray $\mathcal{R}_{\theta}^c$ lands at the characteristic parabolic point of $z_c$, then for any nearby parameter $c' \in H$, the dynamical ray $\mathcal{R}_{\theta}^{c'}$ lands at a dynamical co-root or root on the boundary of the characteristic Fatou component of $f_{c'}$. This proves that $\theta \in S \cup S'$. The details are similar to case (1).
\end{proof}

The fact that all these possibilities are realized will be proven in Section \ref{SecOddBdy}.

Let us begin with an elementary fact about the accumulation set of a parameter ray at a periodic (under the map $t\mapsto -dt$ $($mod $1)$) angle. 

\begin{lemma}\label{parameter_ray_dynamical_ray}
Let $\theta$ be $n$-periodic under the map $t\mapsto -dt$ $($mod $1)$, $\mathcal{R}_{\theta}$ be a parameter ray of $\mathcal{M}_d^*$, and $c \in L_{\mathcal{M}_d^*}(\theta)$. Then, $f_c$ has a parabolic cycle of period $k$ dividing $n$ such that the corresponding dynamical ray $\mathcal{R}_{\theta}^c$ lands at the characteristic parabolic point of $f_c$.
\end{lemma}

\begin{proof}
This is a classical result for the Mandelbrot set \cite[Theorem C.7]{GM} \cite[Proposition 3.1, Proposition 3.2]{Sch}, and can be proved analogously in the antiholomorphic setting.
\end{proof}

\begin{lemma}\label{co-root arcs}
For every co-root arc $\mathcal{C}$ of period $k$, there exists a \emph{unique} $\theta$ of period $k$ (under multiplication by $-d$) such that the dynamical ray $\mathcal{R}_{\theta}^c$ at angle $\theta$ lands at the characteristic parabolic point of $f_c$ for each $c \in \mathcal{C}$. In particular, the parabolic orbit portrait is trivial and constant on every co-root arc. 
\end{lemma}

\begin{proof}
By Corollary \ref{CorArcNeighbors}, there exists a unique hyperbolic component $H$ of odd period $k$ such that $\mathcal{C} \subset \partial H$, and $\mathcal{C}$ does not intersect the boundary of any other hyperbolic component of period $k$. Let $S$ be the set of angles of the dynamical rays that land at the $d$ dynamical co-roots of the characteristic Fatou component of the center $c_0$ of $H$. Therefore, $\vert S \vert = d$. 

Let $\tilde{c} \in \mathcal{C}$. By the definition of co-root arcs, the parabolic periodic points of $f_{\tilde{c}}$ do not disconnect the Julia set of $f_{\tilde{c}}$; hence every parabolic periodic point of $f_{\tilde{c}}$ is the landing point of exactly one periodic dynamical ray. Let $\theta$ be the angle of the unique dynamical ray that lands at the characteristic parabolic point of $f_{\tilde{c}}$. By Lemma \ref{transfer}, $\theta \in S$.

In the dynamical plane of $\tilde{c}$, all the dynamical rays at angles in $\left( S \setminus \lbrace \theta \rbrace \right)$ land at repelling periodic points. Hence under small perturbation of $\tilde{c}$ along the co-root arc, the dynamical rays at these angles continue to land at repelling periodic points (by Lemma \ref{l:preserving-portrait1}). Since every parabolic periodic point must be the landing point of at least one (exactly one on the co-root arcs) periodic dynamical ray, the dynamical ray $\mathcal{R}_{\theta}^c$ must land at the characteristic parabolic point of $f_{c}$ for every parameter $c \in \mathcal{C}$ close to $\tilde{c}$. Therefore, the set $\mathcal{C}^{\prime} := \lbrace c \in \mathcal{C} : \mathcal{R}_{\theta}^c$ lands at the characteristic parabolic point of $f_c \rbrace$
is an open set. Since there are only finitely many choices for $\theta$; $\mathcal{C}$ can be written as the union of finitely many disjoint open sets. It follows from the connectedness of $\mathcal{C}$ that all but one of these open sets are empty. So there exists an angle $\theta$ such that $\mathcal{R}_{\theta}^c$ lands at the characteristic parabolic point of $f_c$ for every $c \in \mathcal{C}$. Now it trivially follows that the parabolic orbit portrait is constant on $\mathcal{C}$. 
\end{proof}

\begin{corollary}\label{singe ray on co-root}
At most one periodic parameter ray can accumulate on a co-root parabolic arc.
\end{corollary}

\begin{proof}
This follows from Lemma \ref{parameter_ray_dynamical_ray} and Lemma \ref{co-root arcs}.
\end{proof}

\begin{lemma}\label{root arcs}
For every root arc $\mathcal{C}$ of period $k$, there exists a \emph{unique} pair of angles $\alpha_1$ and $\alpha_2$ of period $2k$ (under multiplication by $-d$) such that the dynamical rays $\mathcal{R}_{\alpha_1}^c$ and $\mathcal{R}_{\alpha_2}^c$ (and none else) land at the characteristic parabolic point of $f_c$ for each $c \in \mathcal{C}$. In particular, the parabolic orbit portrait is non-trivial and constant on every root arc. 
\end{lemma}

\begin{proof}
By Corollary \ref{CorArcNeighbors}, there exists a unique hyperbolic component $H$ of odd period $k$ such that $\mathcal{C} \subset \partial H$, and $\mathcal{C}$ does not intersect the boundary of any other hyperbolic component of period $k$. Let $S'$ be the set of angles of the dynamical rays that land at the unique dynamical root of the characteristic Fatou component of the center $c_0$ of $H$. By Lemma \ref{primitive}, $\vert S' \vert =2$. 

Let $\tilde{c} \in \mathcal{C}$. By the definition of root arcs, the parabolic periodic points of $f_{\tilde{c}}$ disconnect the Julia set of $f_{\tilde{c}}$; hence every parabolic periodic point of $f_{\tilde{c}}$ is a dynamical root point. Since this root point has odd period $k$, it follows from \cite[Corollary 4.2]{NS} that precisely $2$ dynamical rays land at this root point. Let $\alpha_1$ and $\alpha_2$ be the angles of the dynamical rays that land at the characteristic parabolic point of $f_{\tilde{c}}$. By Lemma \ref{transfer}, $S' = \lbrace\alpha_1, \alpha_2\rbrace$, and hence, the angles $\alpha_1$ and $\alpha_2$ are independent of the choice of $\tilde{c} \in \mathcal{C}$. Therefore, the rays $\mathcal{R}_{\alpha_1}^c$ and $\mathcal{R}_{\alpha_2}^c$ (and none else) land at the characteristic parabolic point of $f_c$ for every $c \in \mathcal{C}$. The statement about orbit portraits now follows easily.
\end{proof}

\begin{corollary}\label{two rays on root}
At most two parameter rays can accumulate on a root parabolic arc.
\end{corollary}

\begin{proof}
This follows from Lemma \ref{parameter_ray_dynamical_ray} and Lemma \ref{root arcs}.
\end{proof}

\begin{lemma}[Where two co-root arcs meet]\label{co-roots meet}
Let $\mathcal{C}_1$ and $\mathcal{C}_2$ be two co-root arcs (of period $k$) such that the dynamical ray $\mathcal{R}_{\theta_1}^c$  at angle $\theta_1$ (respectively $\mathcal{R}_{\theta_2}^c$ at angle $\theta_2$) lands at the characteristic parabolic point of $f_c$ for any parameter $c$ on $\mathcal{C}_1$ (respectively on $\mathcal{C}_2$). Let $\tilde{c}$ be the cusp point where these two arcs meet. Then the two dynamical rays $\mathcal{R}_{\theta_1}^{\tilde{c}}$ and $\mathcal{R}_{\theta_2}^{\tilde{c}}$ land at the characteristic parabolic point of $f_{\tilde{c}}$, and no other ray lands there.
\end{lemma}
\begin{proof}
The dynamical rays $\mathcal{R}_{\theta_1}^{\tilde{c}}$ and $\mathcal{R}_{\theta_2}^{\tilde{c}}$ must land at parabolic points; otherwise by Lemma \ref{l:preserving-portrait1}, they would continue to land at repelling periodic points for nearby parameters contradicting our assumption. To finish the proof, we need to show that these two rays land at the characteristic parabolic point of $f_{\tilde{c}}$: since both these rays have odd period $k$, this would prove that these are the only rays landing there \cite[Theorem 2.6]{Mu}.

If either of the two rays landed at a non-characteristic parabolic point of $f_{\tilde{c}}$, the orbit separation lemma (Lemma \ref{LemOrbitSeparation}) would supply a partition by a pair of rational dynamical rays landing at a common point, stable under small perturbations, separating the critical value from the dynamical ray $\mathcal{R}_{\theta_1}^{\tilde{c}}$ (respectively $\mathcal{R}_{\theta_2}^{\tilde{c}}$). But for parameters on $\mathcal{C}_1$ (respectively $\mathcal{C}_2$) sufficiently close to $\tilde{c}$, the dynamical ray at angle $\theta_1$ (respectively $\theta_2$) lands at the characteristic parabolic point, and there does not exist any rational dynamical ray pair separating the $\theta_1$ (respectively $\theta_2$) ray from the critical value: a contradiction! Thus both $\mathcal{R}_{\theta_1}^{\tilde{c}}$ and $\mathcal{R}_{\theta_2}^{\tilde{c}}$ land at the characteristic parabolic point of $f_{\tilde{c}}$.
\end{proof}

\begin{lemma}[Where a co-root and a root arc meet]\label{root and co-root meet}
Let $\mathcal{C}_1$ and $\mathcal{C}_2$ be a co-root and a root arc (of period $k$) such that the dynamical ray(s) $\mathcal{R}_{\theta_1}^c$ (respectively $\mathcal{R}_{\alpha_1}^c$ and $\mathcal{R}_{\alpha_2}^c$) land(s) at the characteristic parabolic point of $f_c$ for each $c \in \mathcal{C}_1$ (respectively $\mathcal{C}_2$). Let $\tilde{c}$ be the cusp point where these two arcs meet. Then the three dynamical rays $\mathcal{R}_{\theta_1}^{\tilde{c}}$, $\mathcal{R}_{\alpha_1}^{\tilde{c}}$ and $\mathcal{R}_{\alpha_2}^{\tilde{c}}$ land at the characteristic parabolic point of $f_{\tilde{c}}$, and no other ray lands there.
\end{lemma}

\begin{proof}
The proof is similar to that of the previous lemma.
\end{proof}

\section{Boundaries of Odd Period Components}\label{SecOddBdy}

The main goal of this section is to prove Theorem \ref{Exactly d+1}, which describes the structure of the boundaries of hyperbolic components of odd periods of $\mathcal{M}_d^*$.

The basic idea of the proof is to transfer the dynamical co-roots/roots to the parameter plane. Due to the lack of complex analytic parameter dependence, the multiplier map does not extend continuously to the boundary of such a hyperbolic component. Hence, the usual analytic approach is replaced by more combinatorial arguments. Lemma \ref{LemSingleRay}, which can be viewed as a combinatorial rigidity result, lies at the heart of most of our combinatorial arguments.  

Recall that a periodic Fatou component is called \emph{characteristic} if it contains the critical value. We will need the concept of \emph{supporting rays}, which will be defined following the ideas in \cite[Definition 1.3]{Po}. The definition of Poirier is more general, we adapt it to our setting.

\begin{definition}[Supporting Arguments and Rays]\label{Supporting} 
Given a periodic Fatou component $U$ of period $n$ (of $f_c$) and a point $p\in \partial U$ which is fixed by $f_c^{\circ n}$ (i.e. so that $p$ is a dynamical root/co-root of $U$), there are only a finite number of dynamical rays $\mathcal{R}_{\theta_1}^c,\ \mathcal{R}_{\theta_2}^c,\ \cdots,\ \mathcal{R}_{\theta_v}^c$ landing at $p$. These rays divide the plane into $v$ regions. We order the arguments of these rays in counterclockwise cyclic order $\lbrace \theta_1,\ \cdots,\ \theta_v\rbrace$, so that $U$ belongs to the region determined by $\mathcal{R}_{\theta_1}^c$ and $\mathcal{R}_{\theta_2}^c$ ($\theta_1 = \theta_2$ if there is a single ray landing at $p$). The argument $\theta_1$ (respectively the ray $\mathcal{R}_{\theta_1}^c$) is by definition a
(left) supporting argument (respectively the (left) supporting ray) of the Fatou component $U$.
\end{definition}

If $f_{c_0}$ has a periodic critical orbit of odd period $k$, then the characteristic Fatou component $U_{c_0}$ of $f_{c_0}$ has exactly $d$ dynamical co-roots and exactly $1$ dynamical root on its boundary. Let $S =\lbrace \theta_1, \theta_2, \cdots, \theta_d\rbrace$ (each angle of period $k$) be the set of angles of the dynamical rays that land at the $d$ dynamical co-roots on $\partial U_{c_0}$, and let $S' =\lbrace \alpha_1, \alpha_2\rbrace$ (each angle of period $2k$) be the set of angles of the dynamical rays that land at the unique dynamical root on $\partial U_{c_0}$. Furthermore, we have $(-d)^k \alpha_1 = \alpha_2$ and $(-d)^k \alpha_2 = \alpha_1$. Let $\mathcal{P} = \lbrace \mathcal{A}_1, \mathcal{A}_2, \cdots, \mathcal{A}_k \rbrace$ be the orbit portrait associated with the periodic orbit of the unique dynamical root on $\partial U_{c_0}$ such that $ \mathcal{A}_1 = \lbrace \alpha_1 , \alpha_2 \rbrace$. It follows that the dynamical rays $\mathcal{R}_{\alpha_1}^{c_0}$ and $\mathcal{R}_{\alpha_2}^{c_0}$ along with their common landing point separate the characteristic Fatou component from all other periodic (bounded) Fatou components, and $\lbrace \alpha_1 , \alpha_2 \rbrace$ are the characteristic angles of $\mathcal{P}$. Without loss of generality, we can assume that the characteristic arc of $\mathcal{P}$ is $\left(\alpha_1,\ \alpha_2\right)$. According to the definition of supporting rays, each of the dynamical rays $\mathcal{R}_{\theta_1}^{c_0},\ \cdots,\ \mathcal{R}_{\theta_d}^{c_0}$ and $\mathcal{R}_{\alpha_1}^{c_0}$ is a (left) supporting ray for the characteristic Fatou component $U_{c_0}$ of $f_{c_0}$.

To each post-critically finite holomorphic polynomial, one can assign a critical portrait, starting with the choice of preferred (left) supporting rays. A critical portrait is a finite collection of finite subsets of $\mathbb{Q}/\mathbb{Z}$, one subset corresponding to each critical point. This construction is due to Poirier, and we refer the readers to \cite[\S 2]{Po} for the details. For the center $c_0$ of a hyperbolic component of period $k$ of $\mathcal{M}_d^*$, $f_{c_0}^{\circ 2}$ is a post-critically finite polynomial. The choice of a preferred (left) supporting ray (in fact, its angle) on the boundary of the characteristic Fatou component $U_{c_0}$ of $f_{c_0}$ completely determines a critical portrait of $f_{c_0}^2$. Observe that a post-critically finite polynomial may have many critical portraits, depending on the choice of a preferred (left) supporting ray. We will denote a critical portrait associated with the polynomial $f_{c_0}^{\circ 2}$ as $\mathcal{F}_{c_0}$. The choices involved in the construction of a critical portrait adds great flexibility to the following theorem \cite[Theorem 2.4]{Po}, which is a combinatorial way of classifying post-critically finite polynomials. We state the theorem only in the context of the maps $f_{c}^{\circ 2}$ that arise in our setting.
 
\begin{theorem}\cite[Theorem 2.4]{Po}\label{motor}
Let $c_0$ and $c_0'$ be the centers of two hyperbolic components of $\mathcal{M}_d^*$. If for some choices of (left) supporting rays, they have the same critical portraits, i.e., if $\mathcal{F}_{c_0}=\mathcal{F}_{c_0'}$, then $c_0=c_0'$.
\end{theorem}

\begin{lemma}[Supporting Rays Determine Dynamics]\label{LemSingleRay} 
1) Suppose that $f_{c_1}$ and $f_{c_2}$ have periodic critical orbits of \emph{odd} period. Further assume that the dynamical ray $\mathcal{R}_{\theta}^{c_i}$ lands at a dynamical co-root or root of the characteristic Fatou component of $f_{c_i}$, for $i=1,\ 2$. Then, $c_1 = c_2$.

2) Suppose that $f_{c_1}$ and $f_{c_2}$ have parabolic cycles of odd period, and neither of them is a cusp. Assume further that the critical values of $f_{c_1}$ and $f_{c_2}$ have the same Ecalle height, and there is a $\theta$ such that the dynamical rays at angle $\theta$ land at the characteristic parabolic points for both the antiholomorphic polynomials. Then $c_1 = c_2$.
\end{lemma}

\begin{proof}
1) Under the assumptions of the lemma, $\theta$ is a supporting argument for the characteristic Fatou component of $f_{c_i}$, for $i=1,\ 2$. The period of $\theta$ (under multiplication by $-d$) completely determines the period of the critical cycle of $f_{c_1}$ and of $f_{c_2}$. Therefore, the periods of the critical cycles of $f_{c_1}$ and $f_{c_2}$ must be equal, say $k$. This integer $k$ and the (preferred) supporting argument $\theta$ completely determine a critical portrait for $f_{c_i}^{\circ 2}$, for $i=1,\ 2$, and it follows that $f_{c_1}^{\circ 2}$ and $f_{c_2}^{\circ 2}$ have the same critical portraits. By Theorem \ref{motor}, $c_1 = c_2$.

2) If $f_{c_1}$ has a parabolic cycle of period $k$, then $c_1$ lies on the boundary of a hyperbolic component $H$ (with center $c_0$, say) of the same period $k$. By Lemma \ref{transfer}, the dynamical ray $\mathcal{R}_{\theta}^{c_0}$ lands at a dynamical co-root or root point of the characteristic Fatou component of $f_{c_0}$. By part (1), this determines $c_0$ and thus $H$ uniquely. Therefore, $c_1 , c_2 \in \partial H$. Since, none of them is a cusp, $f_{c_1}$ and $f_{c_2}$ have the same rational lamination. 

Since $c_1$ and $c_2$ have identical Ecalle height, their Fatou coordinates provide a conformal conjugacy $\psi^1 :f_{c_1} \sim f_{c_2}$ in the immediate parabolic basins, hence from the interior of the filled-in Julia set $K(f_{c_1})$ to the interior of the filled-in Julia set $K(f_{c_2})$. Since the Julia sets are locally connected, so are the closures of every bounded Fatou component. Therefore, $\psi^1$ extends homeomorphically to the boundary of every bounded Fatou component and conjugates $f_{c_1}$ to $f_{c_2}$.

On the other hand, by the normalized Riemann maps $\phi_{c_i}$ of the basins of infinity $\mathcal{A}_{\infty}\left( f_{c_i} \right)$ constructed in \cite{NS}, we define a conformal conjugacy $\psi^2 = \phi_{c_2}^{-1}\circ\phi_{c_1} : f_{c_1} \sim f_{c_2}$ from the basin of infinity $\mathcal{A}_{\infty}\left( f_{c_1} \right)$ of  $f_{c_1}$ to the basin of infinity$\mathcal{A}_{\infty}\left( f_{c_2} \right)$ of $f_{c_2}$. Since the parabolic Julia set $J(f_{c_i})$ is locally connected, the inverse of the Riemann maps $\phi_{c_i}^{-1}$ extend continuously to $\overline{\mathbb{D}}$. Since $f_{c_1}$ and $f_{c_2}$ have the same rational lamination, the conjugacy $\psi^2$ extends to a homeomorphism of $\overline{\mathcal{A}_{\infty}\left( f_{c_1} \right)}$ onto $\overline{\mathcal{A}_{\infty}\left( f_{c_2} \right)}$ such that it maps the parabolic periodic points (and their inverse orbits) of $f_{c_1}$ to those  of $f_{c_2}$ (in particular, the characteristic parabolic point of $f_{c_1}$ maps to that of $f_{c_2}$). 

It easily follows from our construction that $\psi^1$ and $\psi^2$ agree on a dense subset of their common domains of definition (the union of the boundaries of the bounded Fatou components); namely, on the parabolic cycle and their iterated pre-images. Thus, $\psi^1$ and $\psi^2$ coincide everywhere on their common domains of definition and define a homeomorphism $\psi :\bar{\mathbb{C}} \to \bar{\mathbb{C}}$, which is a topological conjugacy between $f_{c_1}$ and $f_{c_2}$ and conformal outside $J(f_{c_1})$. 

We will construct a conformal conjugacy by a similar argument as in \cite[Lemma 5.8]{HS}. Consider the equipotential $E$ of $f_{c_1}$ at some positive potential, and let $E_1, \cdots, E_k$ be piecewise analytic simple closed curves, one in each bounded periodic Fatou component of $f_{c_1}$, that surround the post-critical set in their Fatou components and that intersect the boundary of their Fatou components in one point, which is on the parabolic cycle. Let $V_0$ be the domain bounded on the outside by $E$ and on the inside by the $E_i$'s. Then there is a quasiconformal homeomorphism $ \psi_{1} : \mathbb{C} \rightarrow \mathbb{C}$ with $\psi_{1} = \psi$ on $\mathbb{C} \setminus V_0$ (i.e., the homeomorphism $\psi$ is modified on $V_0$ so as to become quasiconformal, possibly giving up on the condition that $\psi_{1}$ is a conjugation on $V_0$).

Now construct a sequence of quasiconformal homeomorphisms $\psi_n : \mathbb{C} \rightarrow \mathbb{C}$ as a sequence of pullbacks, satisfying $f_{c_2} \circ \psi_{n+1} = \psi_n  \circ f_{c_1}$ : since the initial conjugacy respects the critical orbits, this construction is possible, and all $\psi_n$ satisfy the same bounds on the quasiconformal dilatation as $\psi_{1}$. Moreover, the support of the quasiconformal dilatation shrinks to the Julia set, which has measure zero. By compactness of the space of quasiconformal maps with a given dilatation, the sequence $\lbrace \psi_n \rbrace$ converges to a quasiconformal conjugation $\psi_{\infty}$ between $f_{c_1}$ and $f_{c_2}$, that is conformal almost everywhere. Hence, $\psi_{\infty}$ is a conformal conjugacy between $f_{c_1}$ and $f_{c_2}$.

Now $f_{c_1}$ is conformally, hence affinely conjugate to $f_{c_2}$. However, since the Riemann maps $\phi_{c_i}$ are normalized, hence tangent to the identity near $\infty$, so is $\psi^2$, and hence the conformal conjugation $\psi_{\infty}$. This implies $f_{c_1} = f_{c_2}$. 
\end{proof}

Using this lemma, we can give an upper bound on the number of parabolic arcs on the boundary of a hyperbolic component of odd period.

\begin{lemma}[Upper Bound on the Number of Parabolic Arcs]\label{upper_bound} 
Every hyperbolic component of odd period, different from $1$, has at most $d$ co-root arcs and at most $1$ root arc on its boundary. The hyperbolic component of period $1$ has exactly $d+1$ co-root arcs on its boundary.
\end{lemma}

\begin{proof}
Let, $H$ be a hyperbolic component of odd period $k \neq 1$. We first give an upper bound on the number of co-root arcs on $\partial H$. 

Let $S =\lbrace \theta_1, \theta_2, \cdots, \theta_d\rbrace$ be the set of angles of the dynamical rays that land at the $d$ dynamical co-roots of the characteristic Fatou component of the center of $H$. Recall that by Lemma \ref{co-root arcs} and its proof, for every co-root arc $\mathcal{C} \subset \partial H$, there exists a \emph{unique} $\theta$ of period $k$ (under multiplication by $-d$) such that $\mathcal{R}_{\theta}^c$ is the only dynamical ray landing at the characteristic parabolic point of $f_c$, for each $c \in \mathcal{C}$. Moreover, $\theta \in S$. Now suppose that there are more than $d$ co-root arcs on $\partial H$. Then there would exist two distinct co-root arcs $\mathcal{C}_1$ and $\mathcal{C}_2$ contained in $\partial H$ such that the dynamical ray $\mathcal{R}_{\theta}^c$ at angle $\theta$ (say) is the unique ray landing at the characteristic parabolic point of $f_c$ for all $c \in \mathcal{C}_1 \cup \mathcal{C}_2$. For $i=1, 2$, let $c_i$ be the parameter on $\mathcal{C}_i$ such that the critical Ecalle height of $c_i$ is $0$. By Lemma \ref{LemArcsDisjoint}, $c_1 \neq c_2$. This contradicts part (2) of Lemma \ref{LemSingleRay}.

Now we turn our attention to the root arcs. Let $S' =\lbrace \alpha_1, \alpha_2\rbrace$ (both angles of period $2k$) be the set of angles of the dynamical rays that land at the unique dynamical root of the characteristic Fatou component of the center of $H$ (see Lemma \ref{primitive}). By Lemma \ref{root arcs} and its proof, for every root arc $\mathcal{C} \subset \partial H$, the dynamical rays $\mathcal{R}_{\alpha_1}^c$ and $\mathcal{R}_{\alpha_2}^c$ (and none else) land at the characteristic parabolic point of $f_c$ for each $c \in \mathcal{C}$. Suppose that $\mathcal{C}_1$ and $\mathcal{C}_2$ are two distinct root arcs contained in $\partial H$. For $i=1, 2$, let $c_i$ be the parameter on $\mathcal{C}_i$ such that the critical Ecalle height of $c_i$ is $0$. Then, the dynamical ray $\mathcal{R}_{\alpha_1}^{c_i}$ would land at the characteristic parabolic point of $f_{c_i}$ for $i=1,2$. By Lemma \ref{LemSingleRay}, $c_1=c_2$. But this contradicts Lemma \ref{LemArcsDisjoint}, which states that two distinct parabolic arcs never intersect. 

Finally, we deal with the hyperbolic component of period $1$. The characteristic Fatou component (which is the unique bounded Fatou component) of the center of the period $1$ hyperbolic component has exactly $d+1$ co-roots and no root. Arguing as above, one easily sees that there are at most $d+1$ co-root arcs on the boundary of this hyperbolic component. By \cite[Corollary 5.10]{NS}, every hyperbolic component of odd period $k$ has a unique center $c_0$ such that the unique critical point $0$ of $f_{c_0}$ has a $k$-periodic orbit. For $k=1$, the only such unicritical antiholomorphic polynomial is $\bar{z}^d$. This proves that there is exactly one hyperbolic component $H_0$ of period $1$ in $\mathcal{M}_d^*$. Now recall that $\mathcal{M}_d^*$ has a $(d+1)-$fold rotational symmetry. More precisely, if $\omega$ is a primitive $(d+1)-$th root of unity, then multiplication by $\omega$ generates a cyclic group (of order $d+1$) of symmetries of $\mathcal{M}_d^*$. In particular, multiplication by $\omega$ yields a $(d+1)-$fold rotational symmetry of the unique hyperbolic component $H_0$ of period $1$ of $\mathcal{M}_d^*$. Therefore, if $\mathcal{C}$ is a parabolic arc on $\partial H_0$, then $\omega\mathcal{C}, \omega^2\mathcal{C}, \cdots, \omega^d\mathcal{C}$ are also distinct parabolic arcs on $\partial H_0$. This shows the existence of $d+1$ parabolic arcs on $\partial H_0$. Since there are at most $d+1$ of them, we conclude that there are exactly $d+1$ co-root arcs on the boundary of the period $1$ hyperbolic component.
\end{proof}

Here is another elementary application of Lemma \ref{LemSingleRay}.

\begin{lemma}
Any two root arcs have disjoint closures.
\end{lemma}

\begin{proof}
Let, $\mathcal{C}_1$ and $\mathcal{C}_2$ be two distinct root arcs. By Lemma \ref{root arcs}, there exist two pairs of angles $\lbrace \alpha_1, \alpha_2\rbrace$ and $\lbrace \alpha_1^{\prime}, \alpha_2^{\prime}\rbrace$ such that the dynamical rays $\mathcal{R}_{\alpha_1}^c$ and $\mathcal{R}_{\alpha_2}^c$ (respectively $\mathcal{R}_{\alpha_1^{\prime}}^c$ and $\mathcal{R}_{\alpha_2^{\prime}}^c$) land at the characteristic parabolic point of $f_c$ for each $c \in \mathcal{C}_1$ (respectively for each $c \in \mathcal{C}_2$). Arguing as in Lemma \ref{upper_bound} (using Lemma \ref{LemSingleRay}), one sees that $\lbrace \alpha_1, \alpha_2\rbrace \cap \lbrace \alpha_1^{\prime}, \alpha_2^{\prime}\rbrace = \emptyset$.

By Lemma \ref{LemArcsDisjoint}, $\mathcal{C}_1 \cap \mathcal{C}_2=\emptyset$. Let, $\tilde{c} \in \overline{\mathcal{C}_1} \cap \overline{\mathcal{C}_2}$. Arguing as in Lemma \ref{co-roots meet}, we see that four distinct dynamical rays $\mathcal{R}_{\alpha_1}^{\tilde{c}},\ \mathcal{R}_{\alpha_2}^{\tilde{c}},\ \mathcal{R}_{\alpha_1^{\prime}}^{\tilde{c}},$ and $\mathcal{R}_{\alpha_2^{\prime}}^{\tilde{c}}$ (each of period $2k$) land at the characteristic parabolic point of $f_{\tilde{c}}$. This contradicts \cite[Lemma 2.10]{Mu}. Hence, $\overline{\mathcal{C}_1} \cap \overline{\mathcal{C}_2} =\emptyset$.
\end{proof}

The next application of Lemma \ref{LemSingleRay} tells us where a parameter ray of odd period must accumulate.

\begin{lemma}\label{LemArcOnBoundary} 
Let $\theta$ be the angle (of period $k$) of the dynamical ray that lands at a dynamical co-root of the characteristic Fatou component of the center $c_0$ of the hyperbolic component $H$ of odd period $k$. Then the parameter ray $\mathcal{R}_{\theta}$ either accumulates on the closure of a single co-root arc of $H$ or lands at a cusp on $\partial H$. 
\end{lemma}

\begin{proof}
Every accumulation point $c$ of $\mathcal{R}_{\theta}$ must have a parabolic cycle of period $k$, and the dynamical ray $\mathcal{R}_{\theta}^c$ must land at the characteristic parabolic point of $f_c$ (by Lemma \ref{parameter_ray_dynamical_ray} and \cite[Lemma 2.5]{Mu}). By Theorem \ref{ThmIndiffBdyHyp}, $c$ lies on the boundary of a hyperbolic component $H'$ of period $k$. By Lemma \ref{transfer}, the center $c'_0$ of $H'$ has the property that the dynamical ray $\mathcal{R}_{\theta}^{c'_0}$ lands at a dynamical co-root of the characteristic Fatou component of $f_{c'_0}$. Now Lemma \ref{LemSingleRay} implies that $c_0 = c_0'$; i.e. $H=H'$. Therefore, the accumulation set of $\mathcal{R}_{\theta}$ is contained in $\partial H$.

Since $\theta$ has odd period $k$, any accumulation point of $\mathcal{R}_{\theta}$ must lie on a co-root arc or a cusp on $\partial H$ (compare Lemma \ref{root arcs}). By the proof of Lemma \ref{upper_bound}, there is at most one co-root arc $\mathcal{C}$ on the boundary of $\partial H$ such that the dynamical ray $\mathcal{R}_{\theta}^c$ lands at the characteristic parabolic point of $f_c$ for each $c \in \mathcal{C}$. Since $L_{\mathcal{M}_d^*}(\theta)$ is connected, this shows that $\mathcal{R}_{\theta}$ either lands at some cusp point on $\partial H$ or accumulates on the closure of a single co-root arc on $\partial H$.
\end{proof}

It follows from Lemma \ref{primitive} that if $c_0$ is the center of a hyperbolic component $H$ of odd period $k$, then exactly two dynamical rays $\mathcal{R}_{\alpha_1}^{c_0}$ and $\mathcal{R}_{\alpha_2}^{c_0}$ (say) land at the dynamical root $z_{c_0}$ of the characteristic Fatou component $U_{c_0}$ of $f_{c_0}$, and both these rays have period $2k$ (under multiplication by $-d$). Further, we have $(-d)^k \alpha_1 = \alpha_2$ and $(-d)^k \alpha_2 = \alpha_1$. Let $\mathcal{P} = \lbrace \mathcal{A}_1, \mathcal{A}_2, \cdots, \mathcal{A}_k \rbrace$ be the orbit portrait associated with the periodic orbit of $z_{c_0}$ such that $ \mathcal{A}_1 = \lbrace \alpha_1 , \alpha_2 \rbrace$. It follows that the dynamical rays $\mathcal{R}_{\alpha_1}^{c_0}$ and $\mathcal{R}_{\alpha_2}^{c_0}$ along with their common landing point separate the characteristic Fatou component from all other periodic (bounded) Fatou components, and $\lbrace \alpha_1 , \alpha_2 \rbrace$ are the characteristic angles of $\mathcal{P}$. Without loss of generality, we can assume that the characteristic arc of $\mathcal{P}$ is $\left(\alpha_1,\ \alpha_2\right)$. This fact as well as the terminologies of this paragraph will be used in the following lemma.

\begin{lemma}\label{rays at dynamical root co-accumulate}
Let $\alpha_1$ and $\alpha_2$ be the angles of the (only) dynamical rays that land at the dynamical root of the characteristic Fatou component of the center of a hyperbolic component $H$ of odd period. Then the parameter rays $\mathcal{R}_{\alpha_1}$ and $\mathcal{R}_{\alpha_2}$ either accumulate on the closure of a common root arc of $H$ or land at a common cusp point on $\partial H$. 
\end{lemma}

\begin{proof}
Let the period of $H$ be $k$. The common period of $\alpha_1$ and $\alpha_2$ under multiplication by $-d$ is $2k$. Any accumulation point of these two parameter rays is either a parabolic parameter of odd period $k$ or a parabolic parameter of even period $r$ with $r \vert 2k$ such that the corresponding dynamical ray of period $2k$ lands at the characteristic parabolic point in the dynamical plane of that parameter (by Lemma \ref{parameter_ray_dynamical_ray} and \cite[Lemma 2.5]{Mu}). Define the set $F$ to be the union of the closure of the finitely many root arcs and cusp points of period $k$, and the finitely many parabolic parameters of even parabolic period and of ray period $2k$ (this is precisely the set of all parameters $c$ for which $f_c$ has a parabolic cycle such that a dynamical ray $\mathcal{R}_{t}^c$ of period $2k$ may land at the parabolic cycle of $f_c$). It follows that the set of accumulation points of $\mathcal{R}_{\alpha_1}$ and $\mathcal{R}_{\alpha_2}$ is contained in $F$. We define $\mathcal{A}_{\mathcal{P}} := \mathcal{A}_1 \cup \cdots \cup \mathcal{A}_k$. 

 Consider the connected components $U_i$ of 
\[
\mathbb{C} \setminus \left(\bigcup_{t \in \mathcal{A}_{\mathcal{P}}} \mathcal{R}_t \cup F \right) 
\;.
\]
There are only finitely many components $U_i$, and they are open and unbounded. Observe that if we define $A:= \lbrace \alpha_1, \alpha_2\rbrace$, then the set $\displaystyle \tilde A:=\bigcup_{j\geq 0}(-d)^{j}A$ coincides with the set $\mathcal{A}_{\mathcal{P}}$.
 
The rest of the proof is similar to that of \cite[Theorem 3.1]{Mi2}. It follows from the discussion before this lemma that $\left(\alpha_1,\alpha_2\right)$ is the characteristic arc of $\mathcal{P}$. Let, $U_1$ be the component which contains all parameters $c$ outside $\mathcal{M}_d^{\ast}$ with external angle $t(c) \in \left( \alpha_1 , \alpha_2 \right)$ (there is such a component as $(\alpha_1 , \alpha_2)$ does not contain any other angle of $\mathcal{A}_{\mathcal{P}}$). $U_1$ must have the two parameter rays $\mathcal{R}_{\alpha_1}$ and $\mathcal{R}_{\alpha_2}$ on its boundary. Pick $c' \in U_1\setminus \mathcal{M}_d^*$ such that $t(c') \in \left( \alpha_1 , \alpha_2 \right)$. By \cite[Lemma 3.4]{Mu}, the dynamical rays $\mathcal{R}_{\alpha_1}^{c'}$ and $\mathcal{R}_{\alpha_2}^{c'}$ land at a common repelling periodic point of $f_{c'}$. Since $U_1 \subset \mathbb{C} \setminus \left(\bigcup_{t \in \mathcal{A}_{\mathcal{P}}} \mathcal{R}_t \cup F \right)$, it follows from Lemma \ref{l:preserving-portrait1} that $\mathcal{R}_{\alpha_1}^{c}$ and $\mathcal{R}_{\alpha_2}^{c}$ land at a common repelling periodic point of $f_{c}$ for all $c$ in $U_1$.

If the two parameter rays $\mathcal{R}_{\alpha_1}$ and $\mathcal{R}_{\alpha_2}$ do not land at the same point or do not accumulate on the closure of the same root arc (recall that every hyperbolic component of odd period has at most one root arc, and any two root arcs have disjoint closures), then $U_1$ would contain parameters $c$ outside $\mathcal{M}_d^{\ast}$ with $t(c) \notin \left( \alpha_1 , \alpha_2 \right)$. But it follows from the remark at the end of the proof of \cite[Theorem 3.1]{Mu} that for such a parameter $c$, the dynamical rays $\mathcal{R}_{\alpha_1}^{c}$ and $\mathcal{R}_{\alpha_2}^{c}$ can not land at a common point, which is a contradiction to the stability of the ray landing pattern throughout $U_i$. Hence, the parameter rays $\mathcal{R}_{\alpha_1}$ and $\mathcal{R}_{\alpha_2}$ must either land at the same parabolic parameter of even (parabolic) period and of ray period $2k$, or accumulate on the closure of the same root arc/cusp of $\mathcal{M}_d^{\ast}$.

Let us assume that $\mathcal{R}_{\alpha_1}$ and $\mathcal{R}_{\alpha_2}$ land at a common parabolic parameter $c$ of even (parabolic) period, and we will arrive at a contradiction. By Lemma \ref{parameter_ray_dynamical_ray}, the dynamical rays $\mathcal{R}_{\alpha_1}^c$ and $\mathcal{R}_{\alpha_2}^c$ of $f_c$ land at the characteristic parabolic point $z_c$ of even period of $f_c$. But, $f_c^{\circ k} (\mathcal{R}_{\alpha_1}^c) = \mathcal{R}_{(-d)^k \alpha_1}^c = \mathcal{R}_{\alpha_2}^c$. As $\mathcal{R}_{\alpha_1}^c$ lands at $z_c$, it follows by continuity that $\mathcal{R}_{\alpha_2}^c$ lands at $f_c^{\circ k}(z_c)$. Thus, $f_c^{\circ k}(z_c)=z_c$. This implies that the period of the parabolic point $z_c$ (under $f_c$) divides the odd integer $k$, a contradiction. 

Therefore, the parameter rays $\mathcal{R}_{\alpha_1}$ and $\mathcal{R}_{\alpha_2}$ either accumulate on the closure of a common root arc of $H'$ or land at a common cusp point on $\partial H'$, for some hyperbolic component $H'$ of odd period $k$. Then, by Lemma \ref{transfer}, the center $c_0'$ (say) of $H'$ has the property that the dynamical rays $\mathcal{R}_{\alpha_1}^{c_0'}$ and $\mathcal{R}_{\alpha_2}^{c_0'}$ land at the dynamical root $z_{c_0'}$ of the characteristic Fatou component $U_{c_0'}$ of $f_{c_0'}$. By Lemma \ref{LemSingleRay}, $c_0 = c_0'$; i.e. $H=H'$. This proves the lemma. 
\end{proof}

\begin{lemma}\label{rays_on_cusps} 
A cusp point can be contained in the accumulation set of at most three parameter rays at periodic (under the map $t\mapsto -dt$ $($mod $1)$) angles.
\end{lemma}

\begin{proof}
Let, $c$ be a cusp point of odd period $k$. If $c$ is contained in the accumulation set of some periodic parameter ray $\mathcal{R}_{\theta}$, then by Lemma \ref{parameter_ray_dynamical_ray}, the dynamical ray $\mathcal{R}_{\theta}^c$ lands at the characteristic parabolic point of $f_c$. The characteristic parabolic point of $f_c$ has odd period $k$. By \cite[Lemma 2.10]{Mu}, at most three periodic dynamical rays can land at a periodic point of odd period. This shows that $c$ can be contained in the accumulation set of at most three parameter rays.
\end{proof}

\begin{lemma}[Lower Bound on the Number of Parabolic Arcs]\label{LemNo.ofArcs} 
Every hyperbolic component $H$ of odd period has at least $3$ parabolic arcs on its boundary. 
\end{lemma}

\begin{proof}
For the hyperbolic component of period $1$, this follows from Lemma \ref{upper_bound}. So we only need to work with the case when the period of $H$ is different from $1$. If $c_0$ is the center of $H$, then $f_{c_0}^{\circ k}$ has exactly $d+1$ fixed points on the boundary of the characteristic Fatou component of $f_{c_0}$. Exactly $d$ of them are the landing points of a single dynamical ray, and the remaining one is the landing point of precisely two rays of period $2k$ (by Lemma \ref{number_root_co_root} and Lemma \ref{primitive}); this yields a total of $d+2$ $(\geq 4)$ rays. Call the set of all parameter rays at these $d+2$ angles $\Omega$. Then, $\vert \Omega\vert \geq 4$. The members of $\Omega$ must accumulate on $\partial H$, by Lemma \ref{LemArcOnBoundary} and \ref{rays at dynamical root co-accumulate}. 

Let $\mathcal{C}$ be a parabolic arc on $\partial H$ with the cusp point $\tilde{c}$ at one of its ends. Assume that the dynamical ray $\mathcal{R}_t^c$ lands at the characteristic parabolic point of $f_c$ for all $c \in \mathcal{C}$. Then by Lemma \ref{co-roots meet} and Lemma \ref{root and co-root meet}, the dynamical ray $\mathcal{R}_t^{\tilde{c}}$ lands at the characteristic parabolic point of $f_{\tilde{c}}$. Therefore, by Lemma \ref{parameter_ray_dynamical_ray}, the set of element(s) of $\Omega$ that can possibly accumulate on a parabolic arc $\mathcal{C}$ on $\partial H$ is a subset of the set of element(s) of $\Omega$ that can possibly accumulate on the cusp points on the ends of the parabolic arc $\mathcal{C}$. At most $3$ members of $\Omega$ can accumulate at any cusp (by Lemma \ref{rays_on_cusps}), at most $2$ on any parabolic arc (Corollaries \ref{singe ray on co-root}, \ref{two rays on root}).

If $\partial H$ contained only a single parabolic arc, then it would contain only a single cusp, and hence, at most $3$ members of $\Omega$ can accumulate on $\partial H$. If $\partial H$ had exactly two parabolic arcs on its boundary, then these would either be both co-root arcs, or a root and a co-root arc (by Lemma \ref{upper_bound}). In both cases, it is again easy to see that at most $3$ members of $\Omega$ can accumulate on $\partial H$. Since $\vert \Omega\vert \geq 4$, it follows that $\partial H$ must contain at least $3$ parabolic arcs.
\end{proof}

\begin{corollary}\label{CorSimpleClosed} 
For every hyperbolic component $H$ of odd period $k$, the boundary $\partial H$ is a simple closed curve.
\end{corollary}

\begin{proof}
By Lemma \ref{LemOddIndiffDyn}, $\partial H$ consists of parabolic arcs and cusps of period $k$. By Corollary \ref{CorArcNeighbors}, if a parabolic arc $\mathcal{C}$ intersects $\partial H$, then $\mathcal{C} \subset \partial H$. Each arc limits on both ends at cusps, by Lemma~\ref{PropArcEnds}. The number of parabolic arcs and cusps on $\partial H$ is finite, by Lemma \ref{upper_bound} and Lemma \ref{LemCuspFin}.

Since $H$ is homeomorphic to $\mathbb{D}$ (by \cite[\S 5]{NS}, the above results imply that $\partial H$ must be a closed curve traversing finitely many parabolic arcs and cusps. Therefore, we can parametrize $\partial H$ by a continuous surjection $\gamma: \left[0,1\right]\rightarrow \partial H$, such that $\gamma(0)=\gamma(1)$. We claim that $\partial H$ is a simple closed curve; if this was not the case, then we would have $t_1 \in \left[0,1\right]$ and $t_2 \in \left(0,1\right)$ with $\gamma(t_1)=\gamma(t_2)$. The rest of the proof will work towards finding a contradiction.

Recall that every parabolic arc is a simple arc (by Theorem \ref{ThmParaArcs}), and two distinct parabolic arcs are disjoint (by Lemma \ref{LemArcsDisjoint}). Therefore, the parameter $\gamma(t_1)$ must be a cusp point. This implies that there are at least three different parabolic arcs (Lemma \ref{LemNo.ofArcs} asserts that $\partial H$ must contain at least $3$ parabolic arcs) meeting at the cusp point $\gamma(t_1)$. Let these three parabolic arcs be $\mathcal{C}_1, \mathcal{C}_2$, and $\mathcal{C}_3$. By Lemma \ref{upper_bound}, at most one of them can be a root arc. We consider two cases.

Case 1 ($\mathcal{C}_1$ is a root arc, and $\mathcal{C}_2$, $\mathcal{C}_3$ are co-root arcs).
Arguing as in Lemma \ref{root arcs} and Lemma \ref{co-root arcs}, we see that there are distinct angles $\alpha_1$, $\alpha_2$, $\theta_1$ and $\theta_2$ such that the dynamical rays $\mathcal{R}_{\alpha_1}^c$ and $\mathcal{R}_{\alpha_2}^c$ land at the characteristic parabolic point of $f_c$ for every $c \in \mathcal{C}_1$, and the dynamical ray $\mathcal{R}_{\theta_i}^c$ lands at the characteristic parabolic point of $f_c$ for every $c \in \mathcal{C}_i$, for $i=2,3$. Furthermore, $\alpha_1$ and $\alpha_2$ have period $2k$, and $\theta_1$ and $\theta_2$ have period $k$. The fact that $\theta_1$ and $\theta_2$ are different, follows from the proof of Lemma \ref{upper_bound}. Now, arguing as in Lemma \ref{co-roots meet} and Lemma \ref{root and co-root meet}, we deduce that four dynamical rays $\mathcal{R}_{\alpha_1}^{\gamma(t_1)},\ \mathcal{R}_{\alpha_2}^{\gamma(t_1)},\ \mathcal{R}_{\theta_1}^{\gamma(t_1)}$ and $\mathcal{R}_{\theta_2}^{\gamma(t_1)}$ land at the characteristic parabolic point (of odd period $k$) of $f_{\gamma(t_1)}$. But this contradicts \cite[Lemma 2.10]{Mu}, which states that at most three periodic dynamical rays can land at a periodic point of odd period.

Case 2 (Each $\mathcal{C}_i$, $i=1, 2, 3$, is a co-root arc).
Arguing as in Lemma \ref{co-root arcs}, we see that there are distinct angles $\theta_1$, $\theta_2$ and $\theta_3$ such that the dynamical ray $\mathcal{R}_{\theta_i}^c$ lands at the characteristic parabolic point of $f_c$ for every $c \in \mathcal{C}_i$, for $i=1, 2, 3$. Furthermore, each $\theta_i$ has period $k$. The fact that all the $\theta_i$ are different, follows from the proof of Lemma \ref{upper_bound}. Now, arguing as in Lemma \ref{co-roots meet}, we deduce that three dynamical rays $\mathcal{R}_{\theta_1}^{\gamma(t_1)},\ \mathcal{R}_{\theta_2}^{\gamma(t_1)}$ and $\mathcal{R}_{\theta_3}^{\gamma(t_1)}$ of odd period $k$ land at the characteristic parabolic point (of odd period $k$) of $f_{\gamma(t_1)}$. But this contradicts \cite[Theorem 2.6]{Mu}.

Hence, $\partial H$ is a simple closed curve. 
\end{proof}

Having proved that the boundary of every hyperbolic component of odd period is a simple closed curve, we have already taken a fundamental step towards the proof of Theorem \ref{Exactly d+1}. In the next few lemmas, we draw some important conclusions about the parabolic arcs and cusp points, which will allow us to give the exact count of parabolic arcs and cusps on the boundary of a hyperbolic component of odd period.

\begin{corollary}[Closure of Arcs Meet Boundary]\label{CorArcClosureBoundary} 
The closure of every parabolic arc $\mathcal{C}$ intersects $\partial \mathcal{M}_d^{\ast}$.
\end{corollary}

\begin{proof}
We give a proof for the co-root arcs, the case of the root arcs is similar. Note that any co-root arc $\mathcal{C}$ lies on the boundary of some hyperbolic component $H$ of odd period $k$ (by Corollary \ref{CorArcNeighbors}). Let $\theta$ be the external angle (of period $k$) of the unique dynamical ray landing at the characteristic parabolic point of $f_c$ for any $c \in \mathcal{C}$ (see Lemma \ref{co-root arcs}). By Lemma \ref{transfer}, the dynamical ray $\mathcal{R}_{\theta}^{c_0}$ lands at a dynamical co-root of the characteristic Fatou component of the center $c_0$ of the hyperbolic component $H$. Hence, by Lemma \ref{LemArcOnBoundary}, the accumulation set $L_{\mathcal{M}_d^*}(\theta)$ of the parameter ray $\mathcal{R}_{\theta}$ must be contained in $\partial H$.

If $\tilde{c} \in L_{\mathcal{M}_d^*}(\theta)$ is not a cusp, then $\tilde{c}$ is on the given parabolic arc $\mathcal{C}$ (by Lemma \ref{parameter_ray_dynamical_ray} and Lemma~\ref{LemSingleRay}). If $\tilde{c} \in L_{\mathcal{M}_d^*}(\theta)$ is a cusp, then there must be two parabolic arcs terminating at $\tilde{c}$ (since $\partial H$ is a simple closed curve consisting of finitely many parabolic arcs and cusp points). Let these two parabolic arcs be $\mathcal{C}_1$ and $\mathcal{C}_2$. By Lemma \ref{co-roots meet} and Lemma \ref{root and co-root meet}, one of these two arcs, say $\mathcal{C}_1$, has the property that the dynamical ray $\mathcal{R}_{\theta}^c$ lands at the characteristic parabolic point of $f_c$ for each $c \in \mathcal{C}_1$. Now it follows from Lemma~\ref{LemSingleRay} that $\mathcal{C}=\mathcal{C}_1$. In either case, $\overline{\mathcal{C}} \bigcap L_{\mathcal{M}_d^*}(\theta) \neq \emptyset$; i.e. the closure of the parabolic arc $\mathcal{C}$ intersects $\partial \mathcal{M}_d^{\ast}$.
\end{proof}

\begin{corollary}[Bifurcations From Odd Periods] \label{PropOddBif}
Every cusp point having a parabolic cycle of odd period $k$ sits on the boundary of a hyperbolic component of period $2k$. 
\end{corollary}

\begin{proof}
Let $\tilde{c}$ be a cusp point of period $k$. By Theorem \ref{ThmIndiffBdyHyp}, there exists a hyperbolic component $H$ of period $k$ such that $\tilde{c} \in \partial H$. Recall that $\partial H$ is a simple closed curve (Corollary \ref{CorSimpleClosed}). Moreover, by Lemma \ref{LemOddIndiffDyn}, $\partial H$ consists only of parabolic arcs and cusp points of period $k$ (more precisely, for each $c \in \partial H$, $f_c$ has a parabolic cycle of period $k$ of multiplicity $1$ or $2$). But since there are only finitely many cusp points of period $k$ (Lemma \ref{LemCuspFin}), there exists a parabolic arc $\mathcal{C}$ of period $k$ (on $\partial H$) limiting at $\tilde{c}$. Now, by Theorem \ref{ThmBifArc}, each end of $\mathcal{C}$ intersects the boundary of a hyperbolic component of period $2k$. Let $H'$ be the hyperbolic component of period $2k$ bifurcating from that end of $\mathcal{C}$ which limits at $\tilde{c}$. Since $\tilde{c}$ is a limit point of $\mathcal{C}$, it follows that $\tilde{c} \in \partial H'$.
\end{proof}

\begin{lemma}[Cusps in Interior]\label{PropCuspsInterior}
Every parabolic cusp is in the interior of $\mathcal{M}_d^{\ast}$.
\end{lemma}

\begin{proof}
If $\tilde{c}$ is a parabolic cusp, then it has a parabolic cycle of odd period
$k$, so it is simultaneously on the boundary of a period $k$ component $W$
by Theorem~\ref{ThmIndiffBdyHyp}, and of a period $2k$ component $W'$ by
Corollary~\ref{PropOddBif}. The dynamical rays landing together at the dynamical
root of the characteristic Fatou component of the centers of  $W$ and  $W'$ are uniquely determined by the
dynamical rays landing at the characteristic parabolic point of $f_{\tilde{c}}$, by Lemma \ref{transfer}.
This determines the parameters at the centers of $W$ and $W'$ uniquely, so
$W$ and $W'$ are the only components with $\tilde{c}$ on its boundary. 

The boundary of $W$ is a simple closed curve by
Corollary~\ref{CorSimpleClosed}. Let $c_1$ and $c_2$ be two points on the two parabolic
arcs $\mathcal{C}_1$ and $\mathcal{C}_2$ respectively that terminate at $\tilde{c}$, and sufficiently close to $\tilde{c}$ so that $c_1,\ c_2 \in \partial W\cap\partial W'$. Connect $c_1$ to $c_2$ by a simple arc $\gamma'$
within $W'$, and by another simple arc $\gamma''$ within $W$ (this is possible since hyperbolic components are homeomorphic to $\mathbb{D}$, by \cite[\S 5]{NS}). Then the curve $\gamma:=\gamma' \cup \gamma''$ is a simple closed curve. We now consider two cases.

\begin{figure}[h]
\centering{\includegraphics[scale=0.40]{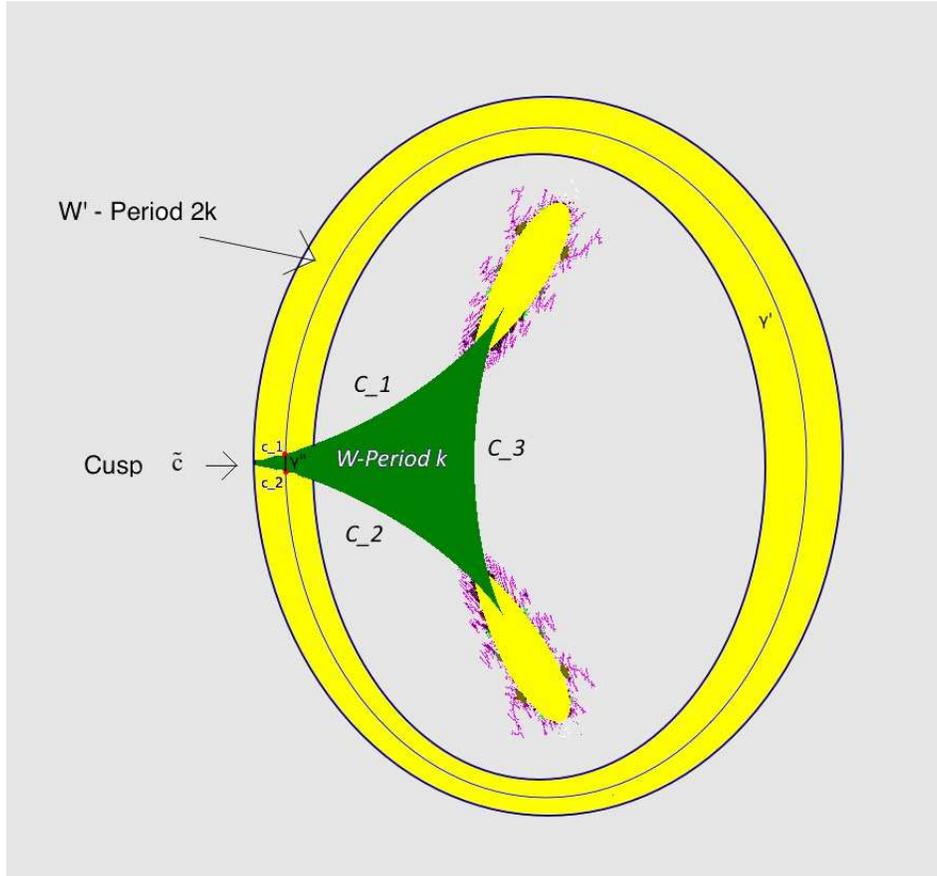}}
\caption{If the simple closed curve $\gamma=\gamma' \cup \gamma''$ does not wind around $\tilde{c}$, then the closure of the parabolic arc $\mathcal{C}_3$ would be contained in the unique bounded component $U$ of $\mathbb{C}\setminus \gamma$, proving that the closure of $\mathcal{C}_3$ would lie in the interior of $\mathcal{M}_d^*$.}
\label{possibility}
\end{figure}

Case 1 (The curve $\gamma$ winds around $\tilde{c}$). 
Note that $\gamma \subset \mathcal{M}_d^*$. Since $\mathcal{M}_d^*$ is a full continuum \cite{Na1}, the unique bounded component $U$ of $\mathbb{C}\setminus \gamma$ is contained in $\mathrm{int}(\mathcal{M}_d^*)$. By our assumption, $\tilde{c} \in U$ . This shows that $\tilde{c} \in \mathrm{int}(\mathcal{M}_d^*)$.

Case 2 (The curve $\gamma$ does not wind around $\tilde{c}$).
If $\gamma$ does not wind around $\tilde{c}$, then $\partial W \cap \partial \mathcal{M}_d^{\ast}=\lbrace \tilde{c}\rbrace$. Since $\partial W$ contains at least $3$ parabolic arcs (by Lemma~\ref{LemNo.ofArcs}), this implies that there exists a third parabolic arc $\mathcal{C}_3 \subset \partial W$ such that $\gamma$ winds around every point $c \in \overline{\mathcal{C}_3}$ (compare Figure \ref{possibility}). Since $\mathcal{M}_d^*$ is a full continuum \cite{Na1}, the unique bounded component $U$ of $\mathbb{C}\setminus \gamma$ is contained in $\mathrm{int}(\mathcal{M}_d^*)$. Therefore, $\overline{\mathcal{C}_3} \subset U \subset \mathrm{int}(\mathcal{M}_d^*)$. But $\overline{\mathcal{C}_3}$ must intersect $\partial \mathcal{M}_d^{\ast}$ (by Corollary \ref{CorArcClosureBoundary}), which is a contradiction to the assumption of case (2).

This shows that every cusp point $\tilde{c}$ is contained in the interior of $\mathcal{M}_d^{\ast}$.
\end{proof}

\begin{corollary}[Arcs Meet Boundary]\label{CorArcsBoundary} 
Every parabolic arc intersects $\partial \mathcal{M}_d^{\ast}$.
\end{corollary}

\begin{proof}
This is clear from Corollary \ref{CorArcClosureBoundary} and Lemma \ref{PropCuspsInterior}.
\end{proof}

\begin{corollary}\label{Odd per par rays} 
Every parameter ray at an angle $\theta$ of odd period $k$ lands/accumulates on a sub-arc of a single co-root parabolic arc of period $k$.
\end{corollary}

\begin{remark}
It is conceivable that the parameter rays accumulating on the parabolic arcs may not land at a single point, they may accumulate on a sub-arc of positive length. This property has been extensively studied in the recent paper \cite{IM}, where a complete description of the accumulation properties of the parameter rays (at rational angles) of the multicorns have been given. One of the principal results of \cite{IM} shows that the accumulation set of every parameter ray accumulating on the boundary of a hyperbolic component of odd period (except period one) of $\mathcal{M}_d^*$ contains an arc of positive length.
\end{remark}

\begin{proof}[Proof of Theorem \ref{Exactly d+1}]
For the hyperbolic component of period $1$, this clearly follows from Lemma \ref{upper_bound} and Corollary \ref{CorSimpleClosed}.

Let $c_0$ be the center of the hyperbolic component $H$ of odd period $k\neq 1$. There are $d+1$ fixed points (under $f_{c_0}^{\circ k}$) on the boundary of the characteristic Fatou component of $f_{c_0}$, $d$ of which are the landing points of one dynamical ray (of period $k$) each, while one is the landing point of exactly two dynamical
rays (each of period $2k$). By Lemmas \ref{LemArcOnBoundary}, \ref{rays at dynamical root co-accumulate} and \ref{PropCuspsInterior}, all the corresponding parameter rays must accumulate on $\partial H$, but not on cusps. Further, two different parameter rays cannot accumulate on the same co-root parabolic arc (Corollary \ref{singe ray on co-root}). Therefore, the $d$ parameter rays at angles of period $k$ accumulate on $d$ distinct co-root arcs of $H$. The two parameter rays at angles of period $2k$ must accumulate on the remaining root arc on $\partial H$ (by Lemma \ref{upper_bound}, there are at most $d+1$ arcs). Thus, there are exactly $d$ co-root arcs and one root arc.
\end{proof}

\begin{corollary}[Bijection Between Dynamical and Parameter Rays]\label{CorBijectionDynPara}
Let $H$ be a hyperbolic component of odd period $k(\neq 1)$ with center $c_0$. The characteristic Fatou component $U_{c_0}$ of $f_{c_0}$ has exactly $d$ dynamical co-roots and exactly $1$ dynamical root on its boundary. Let $S =\lbrace \theta_1, \theta_2, \cdots, \theta_d\rbrace$ (each angle of period $k$) be the set of angles of the dynamical rays that land at the $d$ dynamical co-roots on $\partial U_{c_0}$, and let $S' =\lbrace \alpha_1, \alpha_2\rbrace$ (each angle of period $2k$) be the set of angles of the dynamical rays that land at the unique dynamical root on $\partial U_{c_0}$.

Then, $\partial H$ has exactly $d$ co-root arcs $\mathcal{C}_1, \mathcal{C}_2, \cdots, \mathcal{C}_d$ and exactly $1$ root arc $\mathcal{C}_{d+1}$, such that $L_{\mathcal{M}_d^*}(\theta_i) \subset \mathcal{C}_i$ (for $i \in \lbrace 1, 2, \cdots, d\rbrace$), and $L_{\mathcal{M}_d^*}(\alpha_1) \cup L_{\mathcal{M}_d^*}(\alpha_2) \subset \mathcal{C}_{d+1}$.
\end{corollary}

\begin{corollary}\label{2k determines k}
Let $H$ be a hyperbolic component of odd period $k$. Let $\mathcal{R}_{\alpha_1}$, $\mathcal{R}_{\alpha_2}$ be the two parameter rays at $2k$-periodic angles (under the map $t\mapsto -dt$ $($mod $1)$) that accumulate on the unique root arc, and $\mathcal{R}_{\theta_1}$, $\mathcal{R}_{\theta_2}$ be the two parameter rays at $k$-periodic angles (under the map $t\mapsto -dt$ $($mod $1)$) that accumulate on the two co-root arcs adjacent to the root arc of $H$ such that $\alpha_1 < \theta_1 < \theta_2 < \alpha_2$. Then, $\left( 1+d^k \right) \cdot \left( \theta_1 - \alpha_1 \right) = \left( \alpha_2 - \alpha_1 \right) = \left( 1+d^k \right) \cdot \left( \alpha_2 - \theta_2 \right)$. 
\end{corollary}

\begin{proof}
Let $c_1$ and $c_2$ be the two cusps at the ends of the unique root arc of $H$ and $\mathcal{P}_1$, $\mathcal{P}_2$ be the parabolic orbit portraits of $f_{c_1}$, $f_{c_2}$ respectively. The lengths of the characteristic arcs of $\mathcal{P}_1$, $\mathcal{P}_2$ are $\theta_1 - \alpha_1$ and $\alpha_2 - \theta_2$ respectively (possibly after renumbering $c_1$ and $c_2$). By \cite[Lemma 3.5]{Mu}, $\left( 1+d^k \right) \cdot \left( \theta_1 - \alpha_1 \right) = \left( \alpha_2 - \alpha_1 \right) = \left( 1+d^k \right) \cdot \left( \alpha_2 - \theta_2 \right)$.
\end{proof}

\begin{remark}
We have shown that exactly $2$ rays at $k$-periodic angles and $2$ rays at $2k$-periodic angles (under the map $t\mapsto -dt$ $($mod $1)$) accumulate on the boundary of a hyperbolic component $H$ of odd period $k>1$ of the tricorn. It is easy to see that starting from one of the rays at a $k$-periodic angle accumulating on $\partial H$, one can find the other ray by an algorithm similar to that of finding conjugate rays of the Mandelbrot set (see \cite{BS}). Hence, given any parameter ray accumulating on $\partial H$, this algorithm together with Corollary \ref{2k determines k} allows us to determine all the other rays accumulating there.
\end{remark}

\section{Discontinuity of Landing Points}\label{landingdiscont}
Throughout this section, we assume the following:

\begin{itemize}
\item $H$ is a hyperbolic component of odd period $k$ of $\mathcal{M}_d^*$.

\item $\mathcal{C}_0$ is a root arc on the boundary of $H$.

\item $\mathcal{C}_1$ and $\mathcal{C}_2$ are two co-root arcs on the boundary of $H$ (such that $\mathcal{C}_0$ and $\mathcal{C}_2$ are adjacent to $\mathcal{C}_1$).

\item Let $\alpha_1$ and $\alpha_2$ be the angles of the two dynamical rays which land at the characteristic parabolic point of $f_{c}$ for every parameter $c \in \mathcal{C}_0$. Hence, the two parameter rays at angles $\alpha_1$ and $\alpha_2$ accumulate on $\mathcal{C}_0$.

\item Let $\theta_1$ (respectively $\theta_2$) be the angle of the unique dynamical ray that lands at the characteristic parabolic point of $f_{c}$ for every parameter $c \in \mathcal{C}_1$ (respectively $\mathcal{C}_2$). Hence, the parameter ray at angle $\theta_1$ (respectively $\theta_2$) accumulates on $\mathcal{C}_1$ (respectively $\mathcal{C}_2$).

\item $\mathcal{P}_1$ is the parabolic orbit portrait of the cusp point $f_{c_1}$ where $\mathcal{C}_0$ and $\mathcal{C}_1$ meet such that the characteristic angles are $\lbrace \theta_1 , \alpha_1 \rbrace$.

\item $\mathcal{P}_2$ is the parabolic orbit portrait of the cusp point $f_{c_2}$ where $\mathcal{C}_1$ and $\mathcal{C}_2$ meet such that the characteristic angles are $\lbrace \theta_1 , \theta_2 \rbrace$.

\item $H^{\prime}$ (respectively $H^{\prime \prime}$) is the hyperbolic component of period $2k$ that bifurcates from the arcs $\mathcal{C}_0$ and $\mathcal{C}_1$ (respectively $\mathcal{C}_1$ and $\mathcal{C}_2$).

\item The $\mathcal{P}_1$-wake (respectively $\mathcal{P}_2$-wake) is defined as: $W_{\mathcal{P}_1}$ := The connected component of $\mathbb{C} \setminus \lbrace \mathcal{R}_{\theta_1} \cup \mathcal{R}_{\alpha_1} \cup \mathcal{C}_0 \cup \mathcal{C}_1 \cup \lbrace c_1 \rbrace \rbrace$ not containing $0$. (respectively $W_{\mathcal{P}_2}$ := The connected component of $\mathbb{C} \setminus \lbrace \mathcal{R}_{\theta_1} \cup \mathcal{R}_{\theta_2} \cup \mathcal{C}_1 \cup \mathcal{C}_2 \cup \lbrace c_2 \rbrace \rbrace$ not containing $0$.)

\item For any $\theta \in \mathbb{Q}/\mathbb{Z}$, define the set $S(\theta)$ = $\lbrace c \in \mathbb{C} :  \mathcal{R}_{\theta}^c$ lands $\rbrace$ and the map $L_{\theta} : S(\theta) \mapsto \mathbb{C}$ as $L_{\theta}(c)$ = The landing point of the dynamical ray $\mathcal{R}_{\theta}^c$.
\end{itemize}

\begin{lemma}\label{odd to even combinatorics}
The set of angles of dynamical rays landing at the dynamical root of the characteristic Fatou component of the center of $H^{\prime}$ (respectively $H^{\prime \prime}$) is given by $\lbrace \theta_1 , \alpha_1 , \alpha_2 \rbrace$ (respectively $\lbrace \theta_1 , \theta_2 \rbrace$).
\end{lemma}

\begin{proof}
We prove the result for $H^{\prime}$. The other case can be proved analogously.

The dynamical root $\tilde{z}$ of the characteristic Fatou component of the center $\tilde{c}$ of $H^{\prime}$ has period $k$, and it lies on the boundary of exactly two Fatou components of period $2k$ each. By \cite[Corollary 4.2]{NS}, this dynamical root point $\tilde{z}$ is the landing point either of exactly two rays of period $k$ (which are fixed by the first return map $f_{\tilde{c}}^{\circ k}$ of $\tilde{z}$), or of exactly one ray of period $k$ and exactly two rays of period $2k$ (the first one is fixed; the latter two are
interchanged by the first return map $f_{\tilde{c}}^{\circ k}$ of $\tilde{z}$). Let $\mathcal{A}$ be the set of angles of the dynamical rays of $f_{\tilde{c}}$ landing at $\tilde{z}$. Observe that $\tilde{z}$ is a repelling periodic point, and $H^{\prime}$ does not intersect any parameter ray or any parabolic parameter. Hence, Lemma \ref{l:preserving-portrait1} implies that there exists a real-analytic function $z: H' \rightarrow \mathbb{C}$ such that $z(\tilde{c})=\tilde{z}$, $z(c)$ is the dynamical root of the characteristic Fatou component of $f_c$, and the dynamical rays $\lbrace \mathcal{R}_{t}^{c} : t \in \mathcal{A} \rbrace$ land at $z(c)$, for all $c \in H'$. 

The characteristic parabolic point of $f_{c_1}$ is a triple fixed point of $f_{c_1}^{\circ 2k}$ which splits into two attracting points (both of period $2k$ under $f_{c_1}$) and a repelling point (of period $k$ under $f_{c_1}$) when perturbed into $H^{\prime}$. In fact, this repelling point is the dynamical root point of the characteristic Fatou component of the perturbed antiholomorphic polynomial. If the dynamical rays $\lbrace \mathcal{R}_{\beta}^{c_1} : \beta \in \mathcal{A} \rbrace$ landed at some repelling periodic point(s) of $f_{c_1}$, then for nearby parameters, the dynamical rays at the same angles would continue to land at the real-analytic continuations of those repelling points (by Lemma \ref{l:preserving-portrait1}); which would necessarily be different from the dynamical root point of the characteristic Fatou component: a contradiction. Therefore, the set of rays $\lbrace \mathcal{R}_{\beta}^{c_1} : \beta \in \mathcal{A} \rbrace$ land at the parabolic cycle of $f_{c_1}$. It follows from the orbit separation lemma (Lemma \ref{LemOrbitSeparation}) that these rays must land at the characteristic parabolic point of $f_{c_1}$.

We know by Lemma \ref{root and co-root meet} that the three dynamical rays $\mathcal{R}_{\theta_1}^{c_1}$, $\mathcal{R}_{\alpha_1}^{c_1}$ and $\mathcal{R}_{\alpha_2}^{c_1}$ land at the characteristic parabolic point of $f_{c_1}$, and no other ray lands there. Therefore, $\mathcal{A} \subseteq \lbrace \theta_1$, $\alpha_1$ , $\alpha_2 \rbrace$. It now follows that $\mathcal{A} = \lbrace \theta_1$, $\alpha_1$, $\alpha_2 \rbrace$.
\end{proof}

The external angle $t(c)$ of every parameter $c \in W_{\mathcal{P}_1} \cap \left( \mathbb{C} \setminus \mathcal{M}_d^* \right)$ (respectively $c \in W_{\mathcal{P}_2} \cap \left( \mathbb{C} \setminus \mathcal{M}_d^* \right)$) lies in the characteristic arc $\left( \alpha_1 , \theta_1 \right)$ (respectively $\left( \theta_1 , \theta_2 \right)$) of $\mathcal{P}_1$ (respectively $\mathcal{P}_2$). It follows from the proof of \cite[Theorem 3.1]{Mu} that $f_c$ has a repelling periodic orbit with associated orbit portrait $\mathcal{P}_1$ (respectively $\mathcal{P}_2$) for every $c \in W_{\mathcal{P}_1} \cap \left( \mathbb{C} \setminus \mathcal{M}_d^* \right)$ (respectively $c \in W_{\mathcal{P}_2} \cap \left( \mathbb{C} \setminus \mathcal{M}_d^* \right)$).

\begin{figure}[ht!]
\begin{minipage}{0.48\linewidth}
\centering{\includegraphics[scale=0.08]{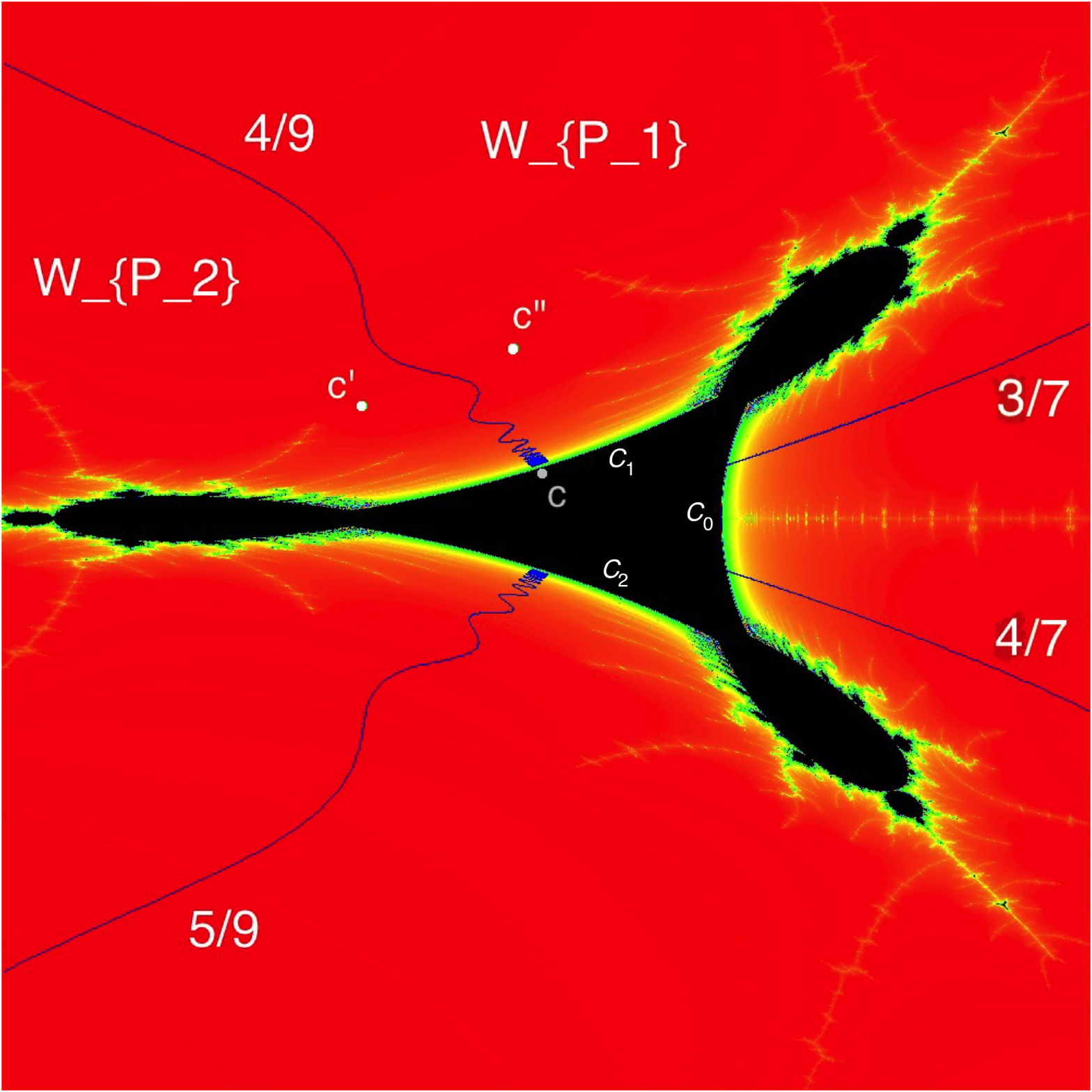}}
\end{minipage}
\begin{minipage}{0.48\linewidth}
\centering{\includegraphics[scale=0.165]{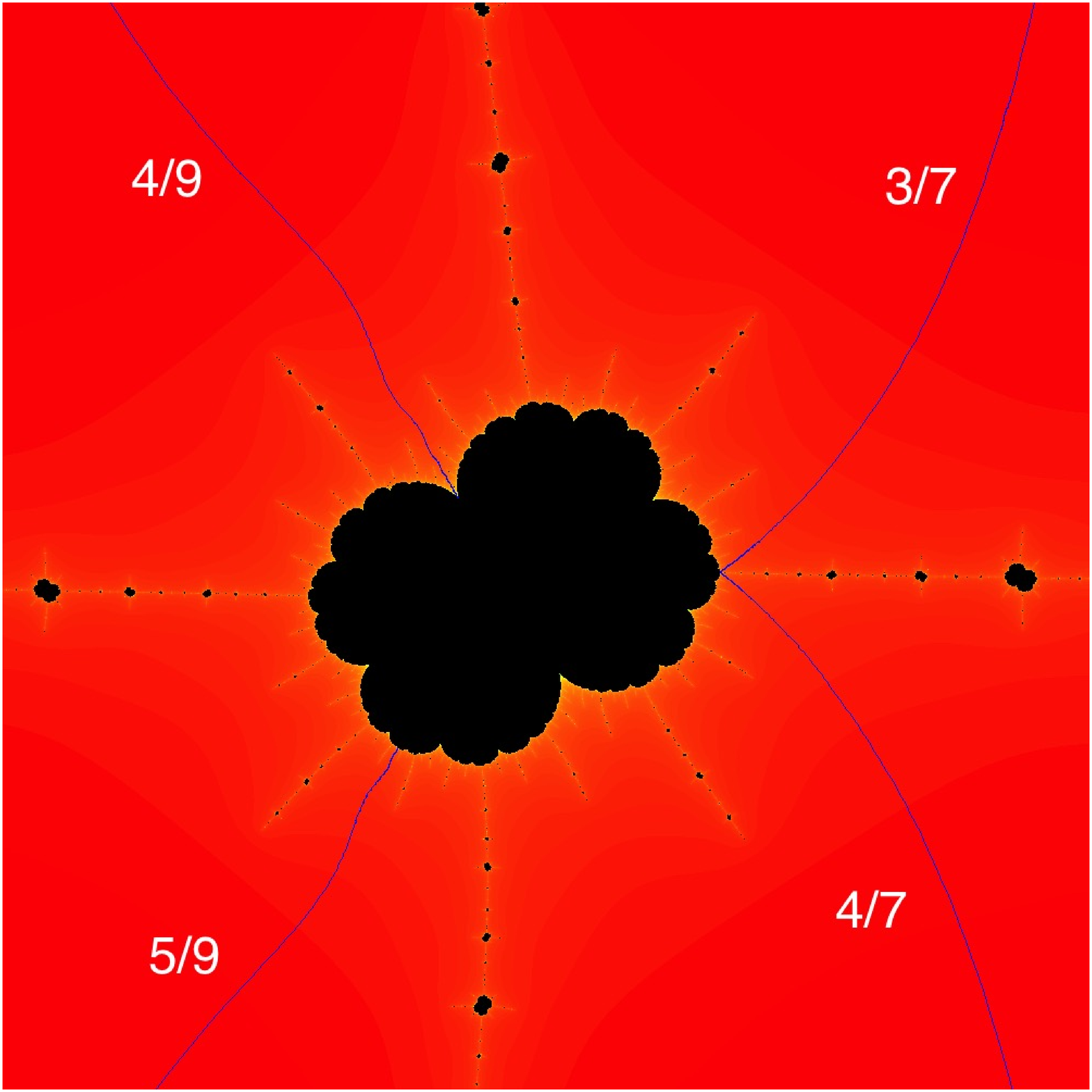}}
\end{minipage}

\begin{minipage}{0.48\linewidth}
\centering{\includegraphics[scale=0.165]{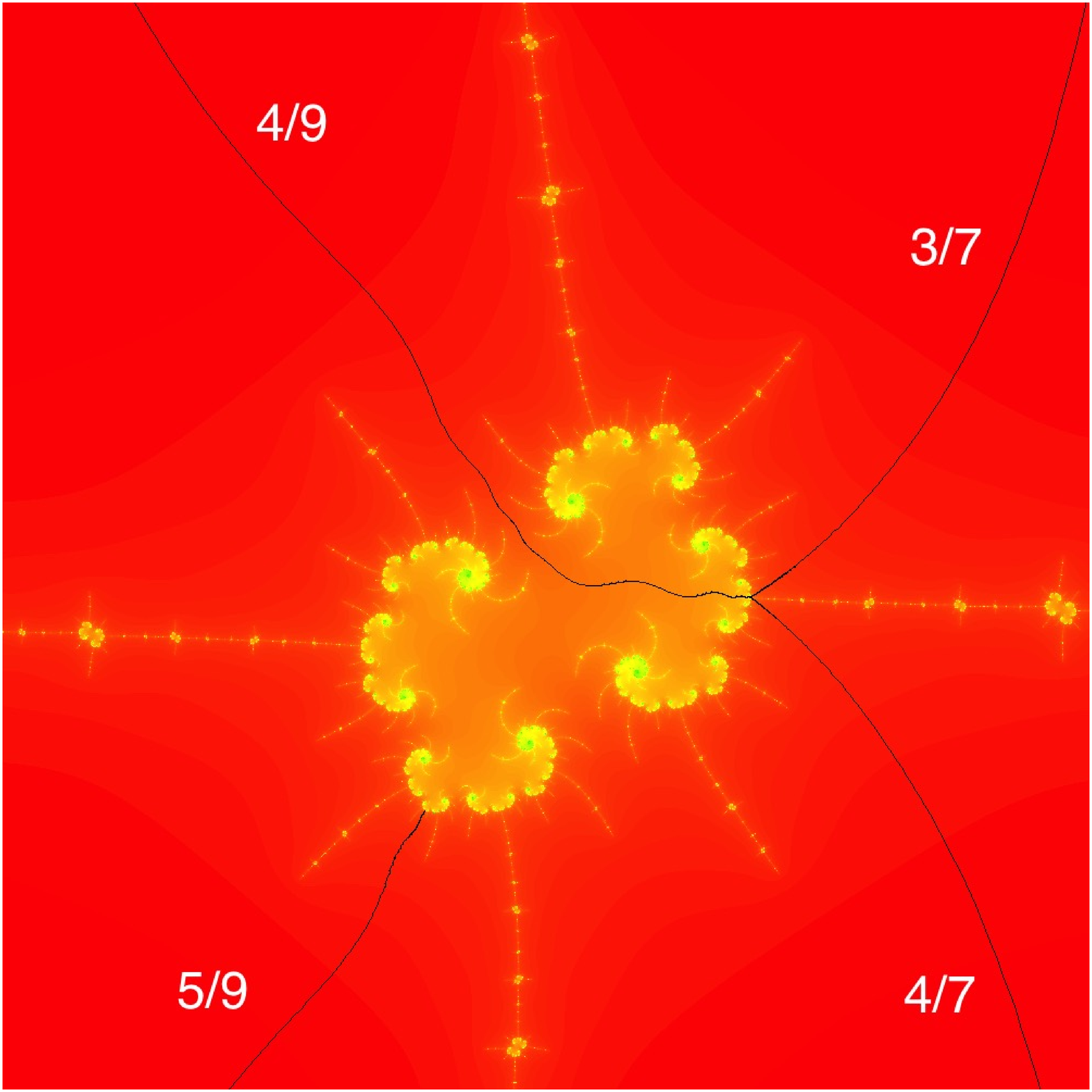}}
\end{minipage}
\begin{minipage}{0.48\linewidth}
\centering{\includegraphics[scale=0.165]{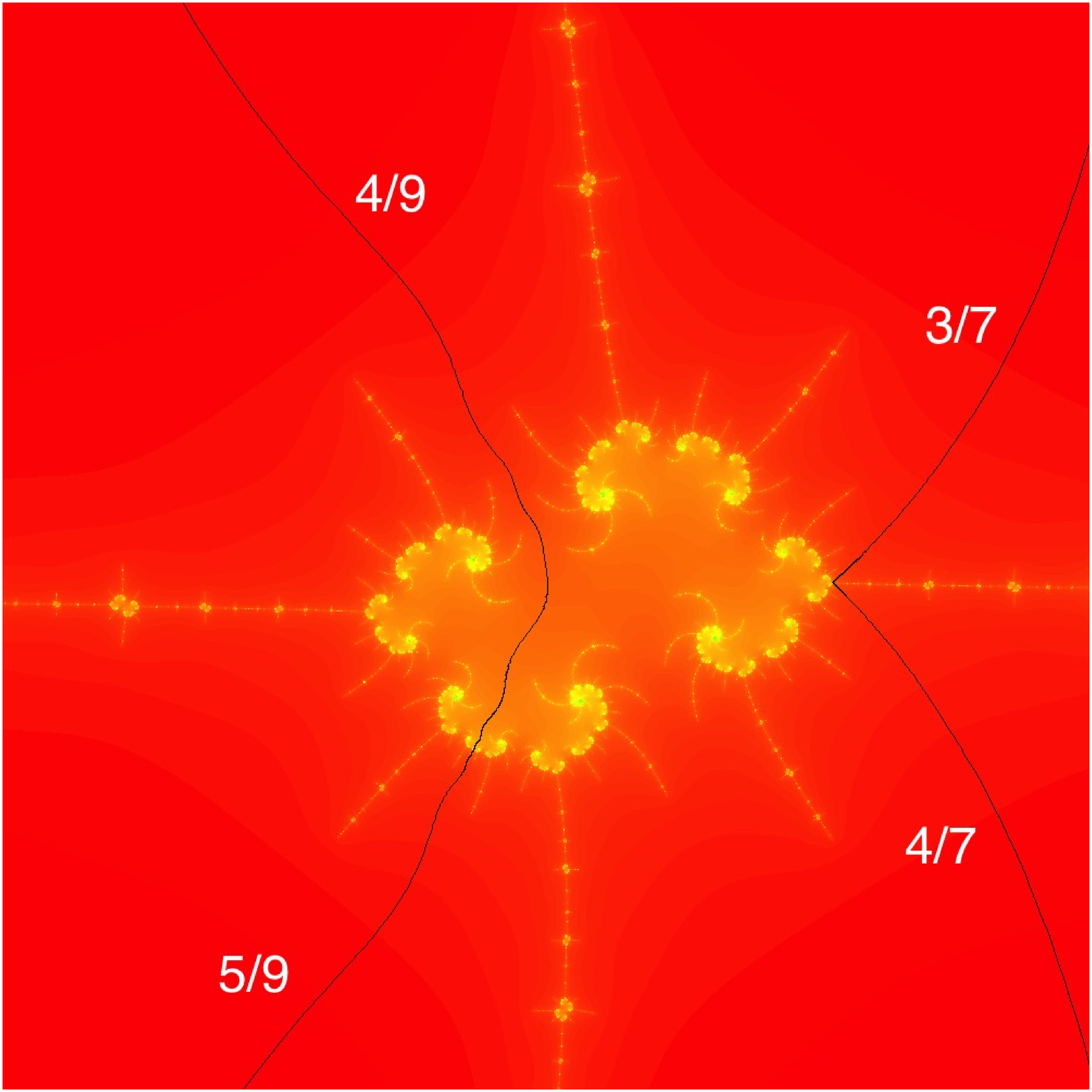}}
\end{minipage}
\caption{Clock-wise from top left: The points $c$, $c'$ and $c''$ lie on the parabolic arc $\mathcal{C}_1$ and in the two wakes $W_{\mathcal{P}_2} $ and $W_{\mathcal{P}_1}$ respectively. The dynamical planes of $f_c$, $f_{c'}$ and $f_{c''}$ exhibit the jump discontinuity of the landing point of the dynamical $4/9$-ray under perturbations into two different wakes.}
\label{landingdiscontinuity}
\end{figure}

The next theorem shows that the landing point of the dynamical ray at angle $\theta_1$ undergoes a jump discontinuity as one approaches the parabolic arc $\mathcal{C}_1$ from the wakes $W_{\mathcal{P}_1}$ and $W_{\mathcal{P}_2}$. The idea is simple: there exist parameters arbitrarily close to $\mathcal{C}_1$ for which the dynamical $\theta_1$-ray and one (or two) other fixed dynamical ray(s) land at a common point, depending on whether the parameter is in the wake $W_{\mathcal{P}_2}$ or $W_{\mathcal{P}_1}$. But on the parabolic arc $\mathcal{C}_1$, which is on the boundary on these wakes, the dynamical $\theta_1$-ray lands alone.

\begin{theorem}[Landing Point Depends Discontinuously on Parameters]\label{discontinuity}
The map $L_{\theta_1} : S(\theta_1) \mapsto \mathbb{C}$ is discontinuous at every point of $\mathcal{C}_1$.
\end{theorem}

\begin{proof}
We first note that the rays $\mathcal{R}_{\theta_2}^c$, $\mathcal{R}_{\alpha_1}^c$ and $\mathcal{R}_{\alpha_2}^c$ land on repelling periodic points for each $c \in \mathcal{C}_1$. Therefore, each of the maps $L_{\theta_2}$, $L_{\alpha_1}$ and $L_{\alpha_2}$ is continuous in a neighborhood of $\mathcal{C}_1$ (this neighborhood does not contain the cusps).

By Theorem \ref{ThmParaArcs} and Lemma \ref{PropArcEnds}, the Ecalle height of the critical value parametrizes the parabolic arc $\mathcal{C}_1$; i.e. there is a real-analytic bijection $c : \mathbb{R} \rightarrow \mathcal{C}_1$ such that $\displaystyle \lim_{t \rightarrow -\infty} c(t) = c_2$ and $\displaystyle \lim_{t \rightarrow \infty} c(t) = c_1$. We assume that $L_{\mathcal{M}_d^*}(\theta_1)=c \left[ a , b \right]$ (with $b \geq a$).

For any $t \in \left( - \infty , b \right]$, every neighborhood $c(t)$ contains parameters $c^{\prime} \in W_{\mathcal{P}_2}$ such that $\mathcal{R}_{\theta_1}^{c^{\prime}}$ and $\mathcal{R}_{\theta_2}^{c^{\prime}}$ land together in the dynamical plane of $f_{c^{\prime}}$; i.e. $L_{\theta_1} \left( c^{\prime} \right) = L_{\theta_2} \left( c^{\prime} \right)$ (see Figure \ref{landingdiscontinuity}), and $L_{\theta_2}$ is continuous in this neighborhood. If $L_{\theta_1}$ was continuous at $c(t)$, this would imply that $L_{\theta_1} \left( c(t) \right) = L_{\theta_2} \left( c(t) \right)$. This contradicts the fact that only one dynamical ray lands at the characteristic parabolic point of $c(t)$ (since $c(t)$ lies on a co-root arc). Therefore, $L_{\theta_1}$ is discontinuous on $c\left( - \infty , b \right]$.

On the other hand, if $t \in \left[ a , \infty \right)$, every neighborhood $c(t)$ would contain parameters $c^{\prime \prime} \in W_{\mathcal{P}_1}$ such that $\mathcal{R}_{\theta_1}^{c^{\prime \prime}}$ and $\mathcal{R}_{\alpha_1}^{c^{\prime \prime}}$ land together in the dynamical plane of $f_{c^{\prime \prime}}$; i.e. $L_{\theta_1} \left( c^{\prime \prime} \right) = L_{\alpha_1} \left( c^{\prime \prime} \right)$ (see Figure \ref{landingdiscontinuity}), and $L_{\alpha_1}$ is continuous in this neighborhood. If $L_{\theta_1}$ was continuous at $c(t)$, we would have $L_{\theta_1} \left( c(t) \right) = L_{\alpha_1} \left( c(t) \right)$. This again contradicts the fact that only one dynamical ray lands at the characteristic parabolic point of $c(t)$ (since $c(t)$ lies on a co-root arc). Therefore, $L_{\theta_1}$ is discontinuous on $c\left[ a , \infty \right)$.
\end{proof}

\begin{remark}
It follows from the proof of Theorem \ref{discontinuity} that the function $L_{\theta_1}$ exhibits a `double' discontinuity on the accumulation set $L_{\mathcal{M}_d^*}(\theta_1)$ of the parameter ray $\mathcal{R}_{\theta_1}$: for any $c \in L_{\mathcal{M}_d^*}(\theta_1) \subset \mathcal{C}_1$, any neighborhood of $c$ contains parameters $c^{\prime}$ and $c^{\prime \prime}$ such that the landing points of the dynamical rays $\mathcal{R}_{\theta_1}^{c'}$ and $\mathcal{R}_{\theta_1}^{c''}$ `jump' in two different directions.
\end{remark}

\begin{corollary}
The set of accumulation points of the parameter ray $\mathcal{R}_{\theta_1}$ is $c\left[ a , b \right] = \lbrace c \in \mathcal{C}_1$ : every neighborhood of $c$ contains $c^{\prime}$ and $c^{\prime \prime}$ such that $L_{\theta_1} \left( c^{\prime} \right) = L_{\theta_2} \left( c^{\prime} \right)$ and $L_{\theta_1} \left( c^{\prime \prime} \right) = L_{\alpha_1} \left( c^{\prime \prime} \right) \rbrace$.
\end{corollary}

\section{Number of Hyperbolic Components}\label{No.ofHypComps}

In this final section, we will prove Theorem \ref{Numberhypcomp}, which gives a formula for the number of hyperbolic components of period $k$ of $\mathcal{M}_d^*$. This is done by counting the number of parameter rays at $k$-periodic angles (under the map $t\mapsto -dt$ $($mod $1)$) that land/accumulate on the boundary of a hyperbolic component of the same period. In the rest of the section, an angle will be called periodic if it is periodic under the map $t\mapsto -dt$ $($mod $1)$.

For any $k \in \mathbb{N} , k > 2$, let $\phi(d,k)$ denote the number of angles in $\mathbb{Q} / \mathbb{Z}$ of period $k$ under multiplication by $-d$. It is easy to check that this is equal to the number of angles in $\mathbb{Q} / \mathbb{Z}$ of period $k$ under multiplication by $d$. It is well-known that the number of hyperbolic components of period $k$ in $\mathcal{M}_d$ is given by $s_{d,k} = \phi(d,k)/d$ (see \cite{EMS}).

\begin{lemma}\label{oddcase}
For any odd $k (\neq 1)$, there are exactly $\phi(d,k) / d$ hyperbolic components of period $k$ in $\mathcal{M}_d^{\ast}$ . Thus, $s^{\ast}_{d,k} = s_{d,k}$.
\end{lemma}

\begin{proof}
We have seen in Section \ref{SecOddBdy} that every parameter ray at a $k$-periodic angle ($k$ odd, $k \neq 1$) lands/accumulates on a sub-arc of a parabolic arc of some hyperbolic component of period $k$. Further, every hyperbolic component of odd period $k$ absorbs exactly $d$ parameter rays at $k$-periodic angles. A single ray can not accumulate on the boundary of two distinct hyperbolic components of period $k$ (the accumulation set of a ray is connected). Therefore, every hyperbolic components of period $k$ has $d$ rays of its own, and this accounts for all the parameter rays at $k$-periodic angles. Hence, $s^{\ast}_{d,k} = \phi(d,k) / d$. The second statement follows from the proof of \cite[Corollary 5.4]{EMS}.
\end{proof}

Now we turn our attention to the hyperbolic components of even periods. The first step is to discuss the number of rays landing on the boundary of an even-periodic hyperbolic component that bifurcates from an odd periodic one. We stick to the terminologies of Section \ref{landingdiscont}.

\begin{lemma}\label{evenrays}
1) Let $k$ be an odd integer. Every rational parameter ray at a $2k$-periodic angle $\theta$ either lands at a parabolic parameter on the boundary of a hyperbolic component of period $2k$ or lands/accumulates on a sub-arc of a parabolic root arc of a hyperbolic component of period $k$.

2) Let $\theta$ be periodic with period $4k$, for some $k \in \mathbb{N}$. Then the parameter ray $\mathcal{R}_{\theta}$ lands at a parabolic parameter of even period. In particular, the landing point lies on the boundary of a hyperbolic component of period $4k$.
\end{lemma}

\begin{proof}
1) Let $\theta$ be of period $2k$ under multiplication by $-d$. If $c$ is an accumulation point of the parameter ray $\mathcal{R}_{\theta}$, then $\mathcal{R}_{\theta}^c$ lands on a parabolic point of period $r$ in the dynamical plane of $f_c$ such that $r \vert 2k$ (by Lemma \ref{parameter_ray_dynamical_ray}). By Lemma \ref{LemIsolated}, there are only finitely many such $c$ with even values of $r$. If $r$ is odd, \cite[Lemma 2.5]{Mu} says that $r = k$. Therefore, the accumulation set of $\mathcal{R}_{\theta}$ is contained in the union of the (finitely many) parabolic arcs of period $k$ and a finite number of parabolic parameters of even period and ray period $2k$. Since the set of accumulation points of a ray is connected, we conclude that $\mathcal{R}_{\theta}$ either lands/accumulates on a sub-arc of a single parabolic (root) arc of a hyperbolic component of period $k$ or lands at a parabolic parameter of even period $r$ with ray period $2k$. In the latter case, if $r = 2k$, it follows from Theorem \ref{ThmIndiffBdyHyp} that the landing point lies on the boundary of some hyperbolic component of period $2k$. On the other hand if $r$ is a proper divisor of $2k$, then it is not hard to show that the multiplier of this $r$-periodic parabolic cycle is a ${2k}/{r}$-th root of unity. It then follows from Theorem \ref{ThmEvenBif} that the landing point lies on the boundary of a hyperbolic component of period $2k$.

2) Every accumulation point of $\mathcal{R}_{\theta}$ is a parabolic parameter of period $r$ with ray period $4k$ such that $r \vert 4k$. By Lemma \cite[Lemma 2.5]{Mu}, $r$ must be even. Since there are only finitely such parabolic parameters (by Lemma \ref{LemIsolated}), the ray must land at a parabolic parameter of ray period $4k$. The fact that the landing point lies on the boundary of a hyperbolic component of period $4k$ is proved as in the previous case.
\end{proof}

\begin{lemma}\label{evenbifurcatesfromodd}
Let $H^{\prime}$ be a hyperbolic component of period $2k$ bifurcating from a hyperbolic component $H$ of odd period $k$ of $\mathcal{M}_d^*$. Then exactly $d-2$ parameter rays at $2k$-periodic angles land on the boundary of $H^{\prime}$.
\end{lemma}

\begin{proof}
If a parameter ray $\mathcal{R}_{\theta}$ at a $2k$-periodic angle lands on the boundary of a hyperbolic component $H^{\prime}$ of period $2k$, then the landing point must be a co-root or root point of $H^{\prime}$ (compare \cite{EMS}). It follows that $\theta$ is contained in the set of angles of the dynamical rays that land at various dynamical root and co-root points on the boundary of the characteristic Fatou component of the center $c$ of $H^{\prime}$. 

The characteristic Fatou component of $f_c$ has exactly $d-2$ dynamical co-root points and exactly one dynamical root point. This dynamical root point is the landing point of exactly two dynamical ray at $2k$-periodic angles and a ray at $k$-periodic angle, by Lemma \ref{odd to even combinatorics}. But the parameter rays at these angles land on the boundary of the odd periodic hyperbolic component $H$. On the other hand, each of the dynamical co-root points is the landing point of exactly one dynamical ray at a $2k$-periodic angle. Therefore, we can assume that $\mathcal{R}_{\theta}^c$ lands at a dynamical co-root of $f_c$. By the previous lemma, the $d-2$ parameter rays at these angles either land at parabolic parameters on the boundary of a hyperbolic components of period $2k$ or land/accumulate on parabolic root arcs of hyperbolic components of period $k$. 

We claim that they land on hyperbolic components of period $2k$. If not, then such a ray at $2k$-periodic angle $\theta$ would accumulate on the root arc of a hyperbolic component $H_1$ of period $k$. Then there is hyperbolic component $H_1^{\prime}$ of period $2k$ (with center $c_1$) bifurcating from $H_1$ such that $\mathcal{R}_{\theta}^{c_1}$ lands at the dynamical root point of the characteristic Fatou component of $f_{c_1}$, and is a (left) supporting ray thereof. This implies that the post-critically finite holomorphic polynomials  $f_c^2$ and $f_{c_1}^2$ have the same critical portrait. It follows from \ref{motor} that $c = c_1$ and hence, $H^{\prime} = H_1^{\prime}$. This is a contradiction since a dynamical ray cannot simultaneously land at a dynamical co-root and a dynamical root point of a Fatou component!

Therefore, all the $d-2$ parameter rays under consideration land on the boundary of hyperbolic components of period $2k$. But it is easy to check that if one of these rays land on the boundary of some other hyperbolic component of period $2k$, then the center of that component would have the same critical portrait as $c$. By \ref{motor}, this is impossible. Hence, all these $d-2$ parameter land on the boundary of $H^{\prime}$ and these are all.
\end{proof} 

\begin{lemma}\label{evennotfromodd}
Let $H$ be a hyperbolic component of even period $k$ that does not bifurcate from an odd period hyperbolic component. Then exactly $d$ parameter rays at $k$-periodic angles land on the boundary of $H$.
\end{lemma}
\begin{proof}
As in the previous lemma, if a parameter ray at an $k$-periodic angle $\theta$ lands on the boundary of  $H$, then $\theta$ is contained in the set of angles of the dynamical rays that land at various dynamical root and co-root points on the boundary of the characteristic Fatou component of the center $c$ of $H$. The characteristic Fatou component of the center of  $H$ has exactly $d-2$ dynamical co-root points and exactly one dynamical root point. Each of these co-roots is the landing point of a unique periodic dynamical ray of period $k$. One can argue as in Lemma \ref{evenbifurcatesfromodd} that the $d-2$ parameter rays at these angles land at parabolic parameters on the boundary of $H$.

The dynamical root point of the characteristic Fatou component of $c$ has a non-trivial orbit portrait. Let its characteristic angles be $\{t^-,t^+\}$. Arguing as in Lemma \ref{rays at dynamical root co-accumulate}, we see that the two parameter rays $\mathcal{R}_{t^-}$ and $\mathcal{R}_{t^+}$ land at a common parabolic parameter on the boundary of a hyperbolic component of period $k$. Another straight-forward application of \ref{motor} shows that this component must be $H$. The landing point of $\mathcal{R}_{t^-}$ and $\mathcal{R}_{t^+}$ is called the root point of $H$, which is the unique parabolic parameter on $\partial H$ with ray period $k$ with the property that the parabolic cycle disconnects the Julia set. The proof of the fact that $\mathcal{R}_{t^-}$ and $\mathcal{R}_{t^+}$ are the only parameter rays landing at this root point is analogous to the holomorphic case (see \cite{EMS}, \cite{Sch}). This completes the proof that exactly $d$ parameter rays at $k$-periodic angles land on the boundary of $H$.
\end{proof}

\begin{proof}[Proof of Theorem \ref{Numberhypcomp}]
For any odd integer $k$, the result follows from Lemma \ref{oddcase}. 

If $k$ is a multiple of $4$, a hyperbolic component of period $k$ can never bifurcate from a hyperbolic component of odd period. In this case, the result follows from Lemma \ref{evennotfromodd}.

In the remaining case when $k$ is twice an odd integer, we have to work a little more. There are exactly  $s^{\ast}_{d,k/2}$ hyperbolic components of period $k/2$ and each of them absorbs two parameter rays at $k$-periodic angles. Each of these hyperbolic components of period $k/2$ have $d+1$ hyperbolic components of period $k$ bifurcating from it. Each of these hyperbolic components of period $k$ absorb exactly $d-2$ parameter rays at $k$-periodic angles. This accounts for $(d+1)s^{\ast}_{d,k/2}$ hyperbolic components of period $k$ and $\lbrace \left( d-2 \right) \left( d+1 \right) + 2 \rbrace  s^{\ast}_{d,k/2} = \left( d^2 - d \right)  s^{\ast}_{d,k/2}$ parameter rays at $k$-periodic angles. We are still left with $ \phi(d,k) - \left( d^2 - d \right)  s^{\ast}_{d,k/2}$ rays at $k$-periodic angles and these must land in groups of $d$ rays on the boundaries of hyperbolic components of period $k$ that do not bifurcate from odd period hyperbolic components (by Lemma \ref{evennotfromodd}). This accounts for $\{\phi(d,k)-(d^2-d)s^{\ast}_{d,k/2}\}/d$ other components of period $k$ and these are all. Combining these, we deduce that
\begin{align*}
s^{\ast}_{d,k} 
&= (d+1)s^{\ast}_{d,k/2} + \{\phi(d,k)-(d^2-d)s^{\ast}_{d,k/2}\}/d 
\\
&= (d+1)s_{d,k/2} + s_{d,k} - (d-1) s_{d,k/2}
\\
&= s_{d,k} + 2 s_{d,k/2} \;.
\end{align*}
\end{proof}

\end{document}